\documentclass[10pt]{amsart}

\usepackage{indentfirst,csquotes}

\usepackage{amsmath}
\usepackage{amsthm}
\usepackage{amssymb}
\usepackage{graphicx}
\usepackage{multicol}
\usepackage[maxfloats=256]{morefloats}

\usepackage{amsthm}

\newtheorem{thm}{\textbf{Theorem}}[section]
\newtheorem{prop}[thm]{\textbf{Proposition}}

\newtheorem{pbm}[thm]{\textbf{Problem}}

\newtheorem{lem}[thm]{\textbf{Lemma}}
\newtheorem*{ack*}{\textbf{Acknowledgement}}

\begin{document}
\title{on elliptic surfaces which have no 1-handles} 
\author[DAISUKE KUSUDA]{DAISUKE KUSUDA}
\footnote[0]{\textit{Date} : \today}
\footnote[0]{2020 \textit{Mathematics Subject Classification}. Primary 57R55, Secondary 57R19.}
\footnote[0]{\textit{Key words and phrases}. 4-manifolds; 1-handles; elliptic surfaces.}
\email{Kusuda-d@ist.osaka-u.ac.jp}
  
\begin{abstract}
Gompf conjectured that the elliptic surface $E(n)_{p,q}$ has no handle decomposition without 1- and 3-handles.
We prove that each of the elliptic surfaces $E(n)_{5,6}$, $E(n)_{6,7}$, $E(n)_{7,8}$ and $E(n)_{8,9}$ has a handle decomposition without 1-handles for $n\geq4$, $n\geq 5$, $n\geq 9$ and $n\geq 24$, respectively.
\end{abstract} 

\maketitle

\section{Introduction}

In 4-dimensional topology, it is an important problem whether the 4-sphere $S^4$ admits an exotic smooth structure.
If an exotic $S^4$ exists, every handle decomposition of it must have at least one 1- or 3-handle (see \cite{Y1}).
Moreover, when a simply connected closed 4-manifold admits a handle decomposition without 1-handles, it has a handle decomposition without 3-handles.
For these reasons, the following problem (Problem 4.18 in Kirby's problem list \cite{K}) is fascinating:
\begin{pbm}
\emph{Does every simply connected, closed 4-manifold have a handlebody decomposition without 1-handles? Without 1- and 3-handles?} 
\end{pbm}

There are many simply connected closed 4-manifolds not to require 1-handles (e.g., \cite{Ru}, \cite{GS}, \cite{AK}, \cite{G}, \cite{Y3}).
On the other hand, Harer, Kas and Kirby (\cite{HHK}) conjectured that every handle decomposition of the elliptic surface $E(1)_{2,3}$ requires at least a 1-handle.
Furthermore, Gompf (\cite{G}) conjectured that the elliptic surface $E(n)_{p,q}$ ($n\geq 1$, \( p,q \geq 2, \gcd(p,q) = 1 \)) has no handle decomposition without 1- and 3-handles.
Here, \( E(n) \) denotes the simply connected elliptic surface with Euler characteristic \( 12n \) and without multiple fibers. 
Then, \( E(n)_{p,q} \) is obtained from \( E(n) \) by logarithmic transformations of the multiplicities $p$ and $q$.

Akbulut (\cite{A}) and Yasui (\cite{Y2}) independently disproved the Harer-Kas-Kirby conjecture.
More strongly, Akbulut proved that $E(1)_{2,3}$ has a handle decomposition without 1- and 3-handles, and Yasui showed that $E(n)_{p,q}$ admits a handle decomposition without 1-handles for $n\geq 1$ and $(p,q)=(2,3), (2,5), (3,4), (4,5)$.  
In addition, Sakamoto (\cite{S}) showed that there exists a handle decomposition of $E(1)_{2,7}$ which has no 1-handle.
In this paper, by modifying Yasui's method, we prove the following theorem: 
\begin{thm}
  \label{thm:theorem1.2}
Each of the elliptic surfaces $E(n)_{5,6}$, $E(n)_{6,7}$, $E(n)_{7,8}$ and $E(n)_{8,9}$ has a handle decomposition without 1-handles for $n\geq4$, $n\geq 5$, $n\geq 9$ and $n\geq 24$, respectively.
\end{thm}
\begin{ack*}
\emph{The author would like to express his deep gratitude to his adviser, Kouichi Yasui, for his encouragement and helpful comments.
He would like to thank to Yohei Wakamaki and Natsuya Takahashi for their encouragement.}
\end{ack*}

\section{Logarithmic transformation}
In this section, we introduce a lemma to prove Theorem \ref{thm:theorem1.2}. 
Let $T$ be a 2-torus embedded in a 4-manifold $X$ with the trivial normal bundle $\nu T\approx T\times D^2$.
\textit{A logarithmic transformation on $T$} is defined as a process of removing the interior of $\nu T$ from $X$ and gluing $T^2\times D^2$ by a diffeomorphism $\varphi : T^2\times \partial D^2 \rightarrow \partial \nu T$.  
Furthermore, for a fixed projection $\pi : \nu T \rightarrow D^2$, the absolute value of the degree of the map $\pi \circ \varphi | _{\{\textit{pt.}\}\times \partial D^2}:\{\textit{pt.}\}\times \partial D^2\rightarrow \partial D^2$ is called the \textit{multiplicity} of the logarithmic transformation. 

Now, suppose that $T\subset X$ \textit{lies in a cusp neighborhood}, \textit{i.e.}, there exists a cusp neighborhood $N$ (shown in Figure \ref{fig:cusp_neighborhood}) embedded in $X$ such that $T$ is a regular fiber of $N$ and an elliptic fibration on $N$ determines $\pi : \nu T\rightarrow D^2$. 
The diffeomorphism type of $X_{p_1,p_2, \cdots ,p_k}$ depends only on $X$, $T$, $\pi$ and the multiplicities.
(Here, we denote $X_{p_1,p_2,\cdots, p_k}$ by the 4-manifold obtained from $X$ by logarithmic transformations on parallel copies of $T$ with multiplicities $p_1, p_2, \cdots, p_k$.)
\begin{figure}[htbp]
\begin{center}
  \includegraphics[width=5cm]{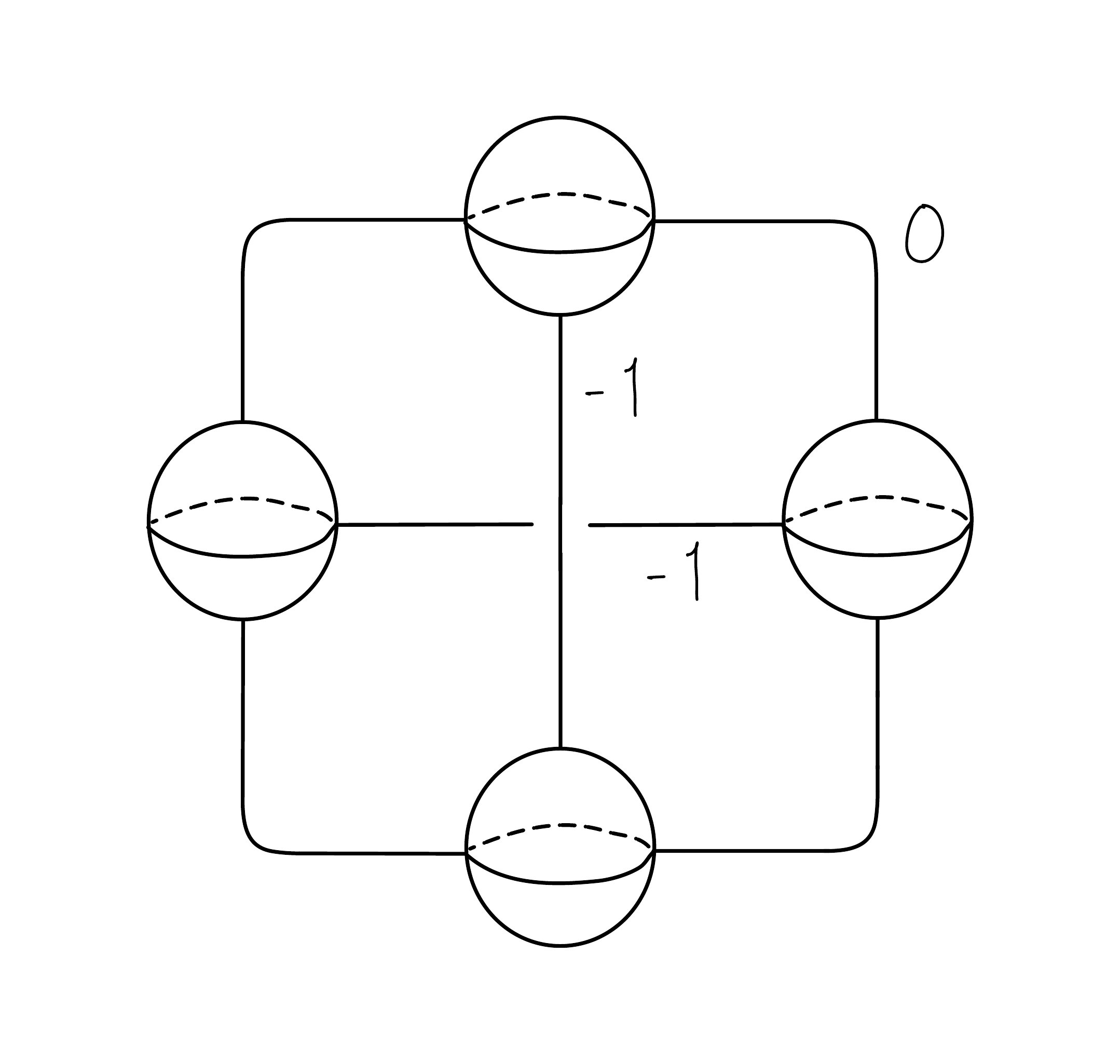} 
 \caption{Cusp neighborhood}
  \label{fig:cusp_neighborhood}
\end{center}
\end{figure}
In particular, the elliptic surface $E(n)_{p_1,p_2,\cdots,p_k}$ is determined by $n\in \mathbb{N}$ and the multiplicities $p_1, p_2, \cdots, p_k$ up to diffeomorphism.
For more details on the logarithmic transformation, see also \cite{E}, \cite{G}, \cite{GS} and \cite{HHK}.

Yasui (\cite{Y2}) proved the following lemma related to the above operation:

\begin{lem}
\label{thm:lemma2.1}
Let $X$ be a simply connected closed 4-manifold with a handle decomposition as shown in Figure \ref{fig:diagram-2.1}.
In this figure, $p$ is an arbitrary integer greater than 1, $q$ is an arbitrary integer and $h_2$, $h_3$ are arbitrary non-negative integers.
Furthermore, let $X_p$ be a 4-manifold obtained by the logarithmic transformation on a 2-torus lying in the cusp neighborhood with the multiplicity $p$.
Then, $X_{p}$ admits a handle decomposition without 1-handles.

  \
\begin{figure}[htbp]
\begin{center}
  \includegraphics[width=10cm]{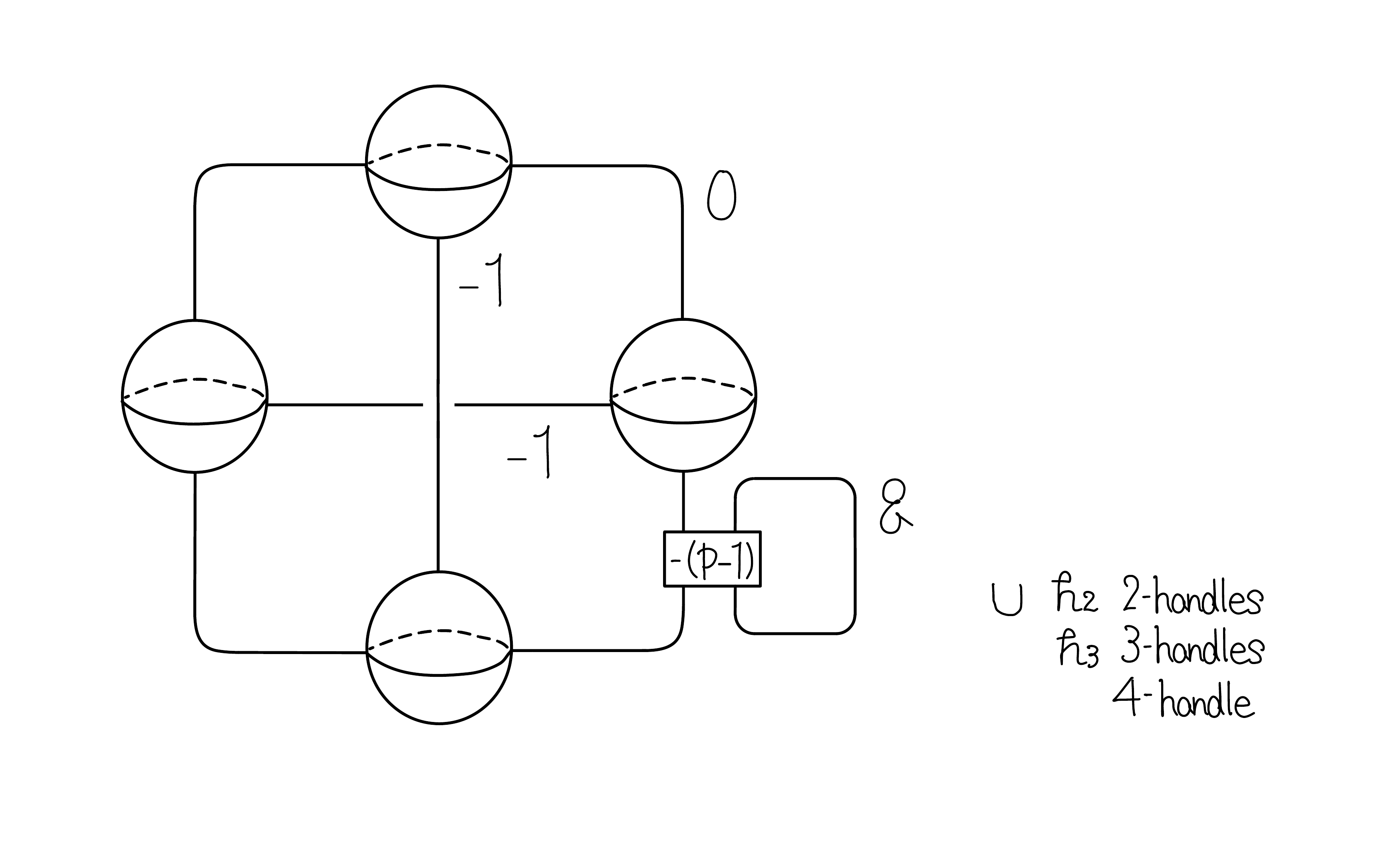} 
 \caption{}
  \label{fig:diagram-2.1}
\end{center}
\end{figure}
\
\end{lem}
  
\section{Proof of theorem 1.2}

First, we give some propositions to prove Theorem \ref{thm:theorem1.2}.
Note that throughout this section, the framing coefficient of each 2-handle we don't make explicitly is $-1$.

\begin{prop}
\label{prop:prop-3.1}
  For \(n \geq 4\), the elliptic surface $E(n)_5$ admits a handle decomposition in Figure \ref{fig:prop.1}. In this figure, the cusp neighborhood that contains the red 2-handles is isotopic to a tubular neighborhood of a cusp fiber in $E(n)_5$.
\end{prop}

\begin{figure}[htbp]
\begin{center}
  \includegraphics[width=11cm]{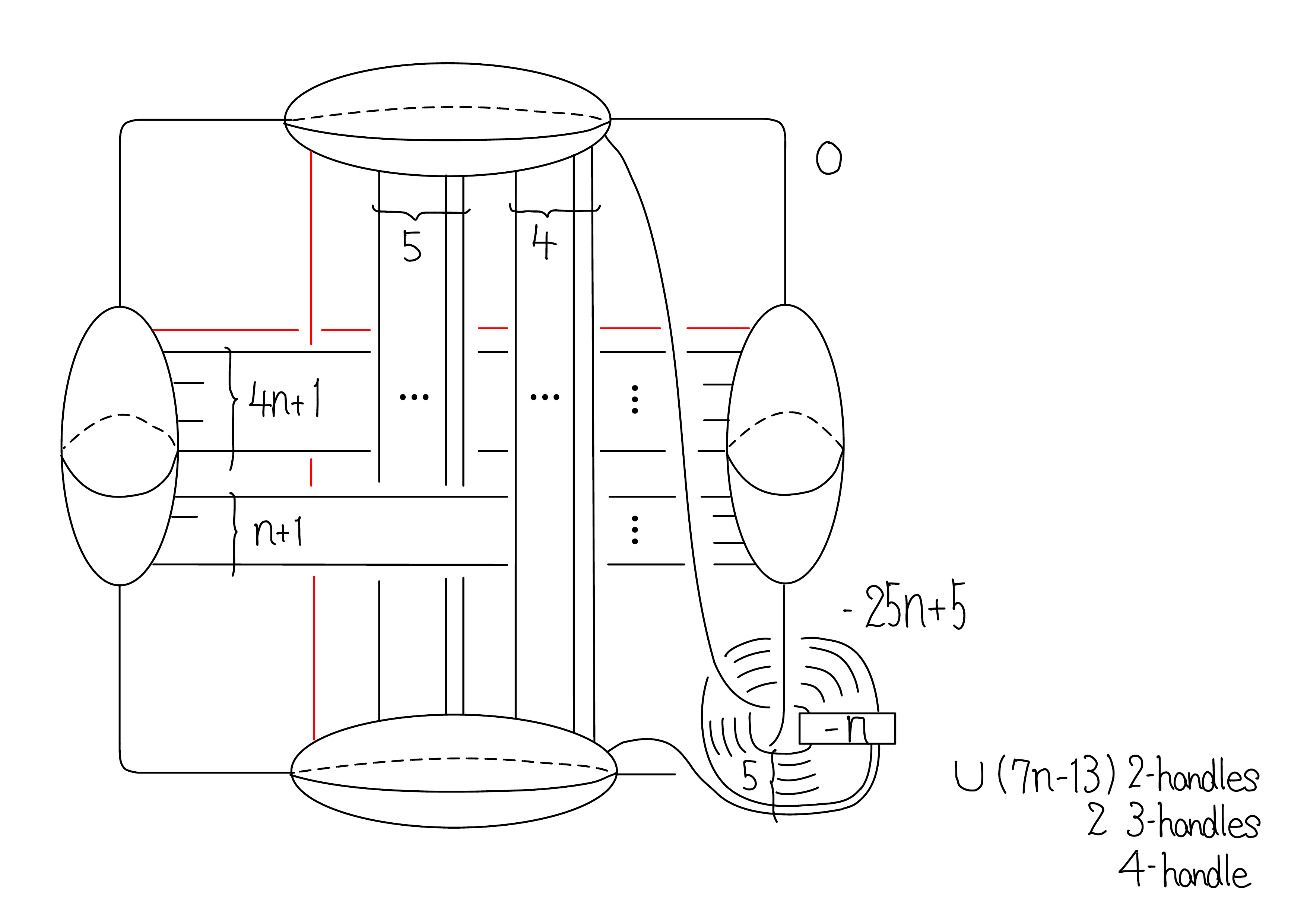} 
 \caption{}
 \label{fig:prop.1}
\end{center}
\end{figure}

\begin{proof}
  
The elliptic surface $E(n)_{p}$ is given in Figure \ref{fig:diagram-0} (see \cite{GS}), and we begin with Figure \ref{fig:diagram-1}, that is, the $p=5$ case of Figure \ref{fig:diagram-0}.
Note that in the proof of this proposition, we do not slide the 1-handles, the $0$-framed 2-handle and the red 2-handles in Figure \ref{fig:diagram-1}. 
First, we focus on Figure \ref{fig:diagram-3} (a) which is a part of Figure \ref{fig:diagram-1}.
We slide one vertical 2-handle over the other, and isotope the diagram to get the bottom-left picture of Figure \ref{fig:diagram-3}.
Sliding the $(-2)$-framed unknot over the horizontal $(-1)$-framed knot, we obtain the bottom-right.
We isotope the new $(-1)$-framed knot and get Figure \ref{fig:diagram-3} (b).
This deformation is introduced by Yasui in his paper \cite{Y2}.
Then, we can regard it as a conversion of the word $aba$ into the word $bab$ (\textit{i.e.}, we can consider each vertical (\textit{resp.} horizontal) 2-handle in Figure \ref{fig:diagram-3} (a) and (b) to be the letter $a$ (\textit{resp.} $b$)).
Furthermore, this procedure is reversible and we can change Figure \ref{fig:diagram-3} (b) into Figure \ref{fig:diagram-3} (a) by handle slides and isotopy.
Next, using the conversions $aba\rightarrow bab$ and $bab\rightarrow aba$, we change the word $(ab)^{6n}$ and discover the word $a^4b^{n+1}a^5b^5(ab)^{5n-9}$ for the following conversions:
\begin{align*}
 abababababab\underbrace{ab\cdots ab}_{12n-12}
&\rightarrow 
  aabaabbabbab\underbrace{ab\cdots ab}_{12n-12}\\
&=
aabaabbabbab\underbrace{ab\cdots ab}_{2n-8}ab\underbrace{abababababab}_{12}\underbrace{ab\cdots ab}_{10n-18}\\
&\rightarrow 
aa\underline{b}aabb\underline{a}bb\underline{a}b\underbrace{\underline{a}b\cdots \underline{a}b}_{2n-8}a\underline{b}\underbrace{aa\underline{b}aabb\underline{a}bb\underline{a}b}_{12}\underbrace{ab\cdots ab}_{10n-18}\\
&\sim
aaaa\underbrace{bb\cdots b}_{n+1}\underbrace{aaaaa}_{5}\underbrace{bbbbb}_{5}\underbrace{ab\cdots ab}_{10n-18}
\end{align*}
Here, we omit underlined letters in the last conversion.
Therefore, in Figure \ref{fig:diagram-1}, we perform handle slides and isotopies indicated by these conversions and get Figure \ref{fig:diagram-6}.
In addition, we can convert the word $(ab)^5$, which indicates the left diagrams of Figure \ref{fig:diagram-7} as follows:  
\begin{align*}
ababababab
&\rightarrow
aabababbab
&
ababababab
&\rightarrow
abaabababb\\
&\rightarrow
aaabaababa
&
&\rightarrow
bababbabbb\\
&\rightarrow
aaabaaabaa
&
&\rightarrow
bbabbbabbb
\end{align*}
Consequently, we get the right diagrams in Figure \ref{fig:diagram-7}.
Note that we use the conversion $(ab)^5\rightarrow (a^3b)^2a^2$ not in this proof but in the proof of Proposition \ref{prop:prop-3.3}.   
Finally, in Figure \ref{fig:diagram-6}, we repeat the deformation in Figure \ref{fig:diagram-7} (b) $(n-2)$ times, so that we can change $(5n-10)$ vertical 2-handles and $(5n-10)$ horizontal ones into $(2n-4)$ vertical ones and $(8n-16)$ horizontal ones.  
Then, we omit $(2n-4)$ vertical and $(4n-12)$ horizontal knots with $(-1)$ framings and get Figure \ref{fig:prop.1}.  
\end{proof}

\begin{prop}
 \label{prop:prop-3.2}
  For \(n \geq 5\), the elliptic surface $E(n)_6$ admits a handle decomposition in Figure \ref{fig:prop.2}. In this figure, the cusp neighborhood that contains the red 2-handles is isotopic to a tubular neighborhood of a cusp fiber in $E(n)_6$.
\end{prop}

\begin{figure}[htbp]
\begin{center}
  \includegraphics[width=10cm]{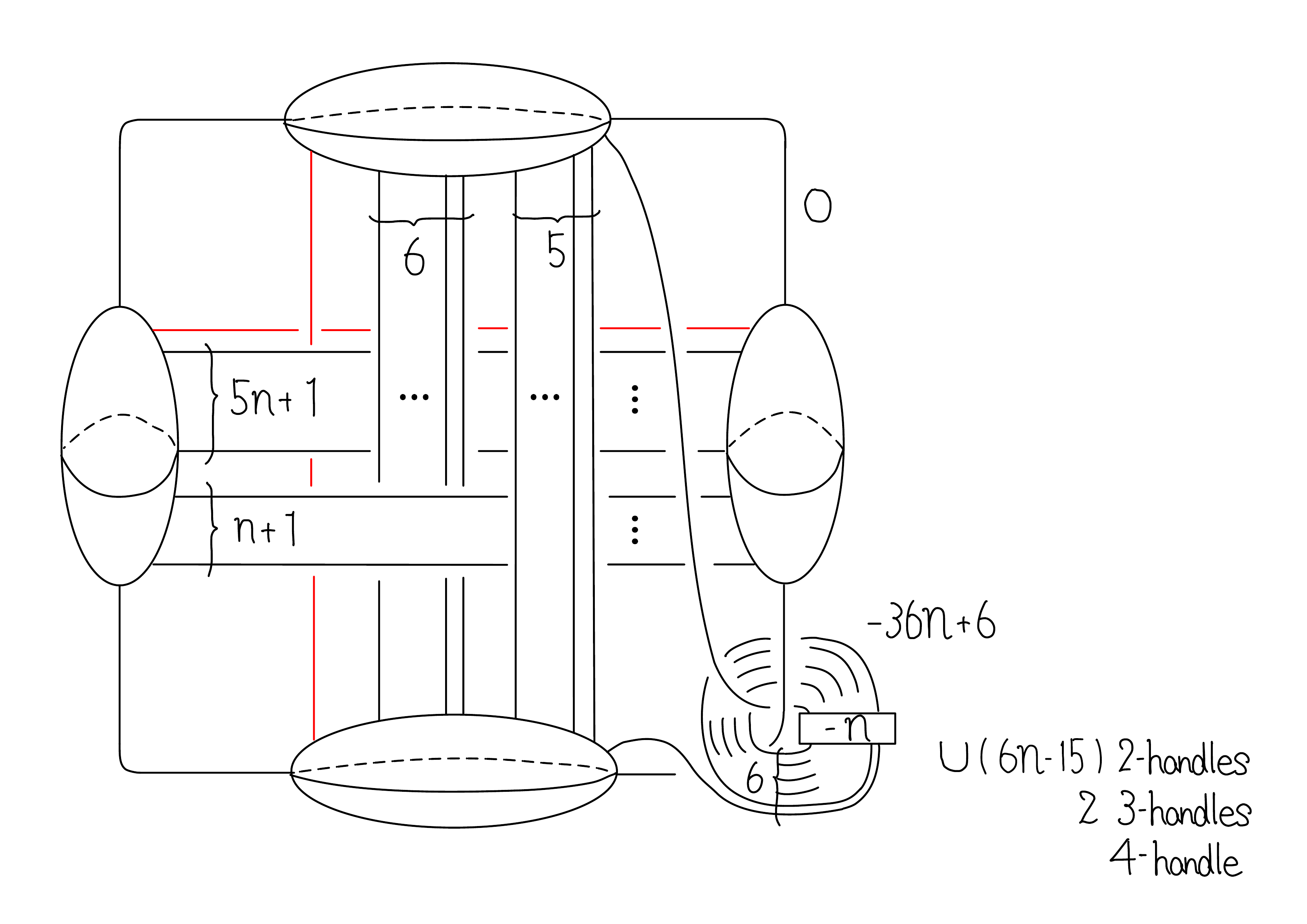} 
 \caption{}
 \label{fig:prop.2}
\end{center}
\end{figure} 

\begin{proof}
We begin with Figure \ref{fig:diagram-8}, that is, the $p=6$ case of Figure \ref{fig:diagram-0}. 
Note that in the proof of this proposition, we do not slide the 1-handles, the 0-framed 2-handle and the red 2-handles in Figure \ref{fig:diagram-8}. 
Using the conversions $aba\rightarrow bab$ and $bab\rightarrow aba$, we change the word $(ab)^{6n}$ and discover the word $a^5b^{n+1}a^6b^2(ab)^{5n-8}$ for the following conversions:
\begin{align*}
abababababab\underbrace{ab\cdots ab}_{12n-12}
&\rightarrow
aabababbabab\underbrace{ab\cdots ab}_{12n-12}\\
&\rightarrow
aaabaabbbabb\underbrace{ab\cdots ab}_{2n-8}\underbrace{abababababab}_{12}\underbrace{ab\cdots ab}_{10n-16}\\
&\rightarrow
aaabaabbbabb\underbrace{ab\cdots ab}_{2n-8}\underbrace{aabababbabab}_{12}\underbrace{ab\cdots ab}_{10n-16}\\
&\rightarrow
aaa\underline{b}aabbb\underline{a}bb\underbrace{\underline{a}b\cdots \underline{a}b}_{2n-8}\underbrace{aaa\underline{b}aa\underline{bbb}abb}_{12}\underbrace{ab\cdots ab}_{10n-16}\\
&\sim
\underbrace{aaaaa}_{5}\underbrace{b\cdots b}_{n+1}\underbrace{aaaaaa}_{6}bb\underbrace{ab\cdots ab}_{10n-16}\\
\end{align*}
Here, we omit underlined letters in the last conversion.
Therefore, in Figure \ref{fig:diagram-8}, we perform handle slides and isotopies indicated by these conversions and obtain Figure \ref{fig:diagram-13}.
In Figure \ref{fig:diagram-13}, we repeat the deformation in Figure \ref{fig:diagram-7} (b) $(n-2)$ times, so that we can change $(5n-10)$ vertical 2-handles and $(5n-10)$ horizontal ones into $(2n-4)$ vertical ones and $(8n-16)$ horizontal ones.  
Then, we omit $(2n-3)$ vertical and $(3n-14)$ horizontal knots with $(-1)$ framings and get Figure \ref{fig:prop.2}.  
\end{proof}

\begin{prop}
\label{prop:prop-3.3}
  For \(n \geq 9\), the elliptic surface $E(n)_7$ admits a handle decomposition in Figure \ref{fig:prop.3}. In this figure, the cusp neighborhood that contains the red 2-handles is isotopic to a tubular neighborhood of a cusp fiber in $E(n)_7$.
\end{prop}

\begin{figure}[htbp]
\begin{center}
  \includegraphics[width=10cm]{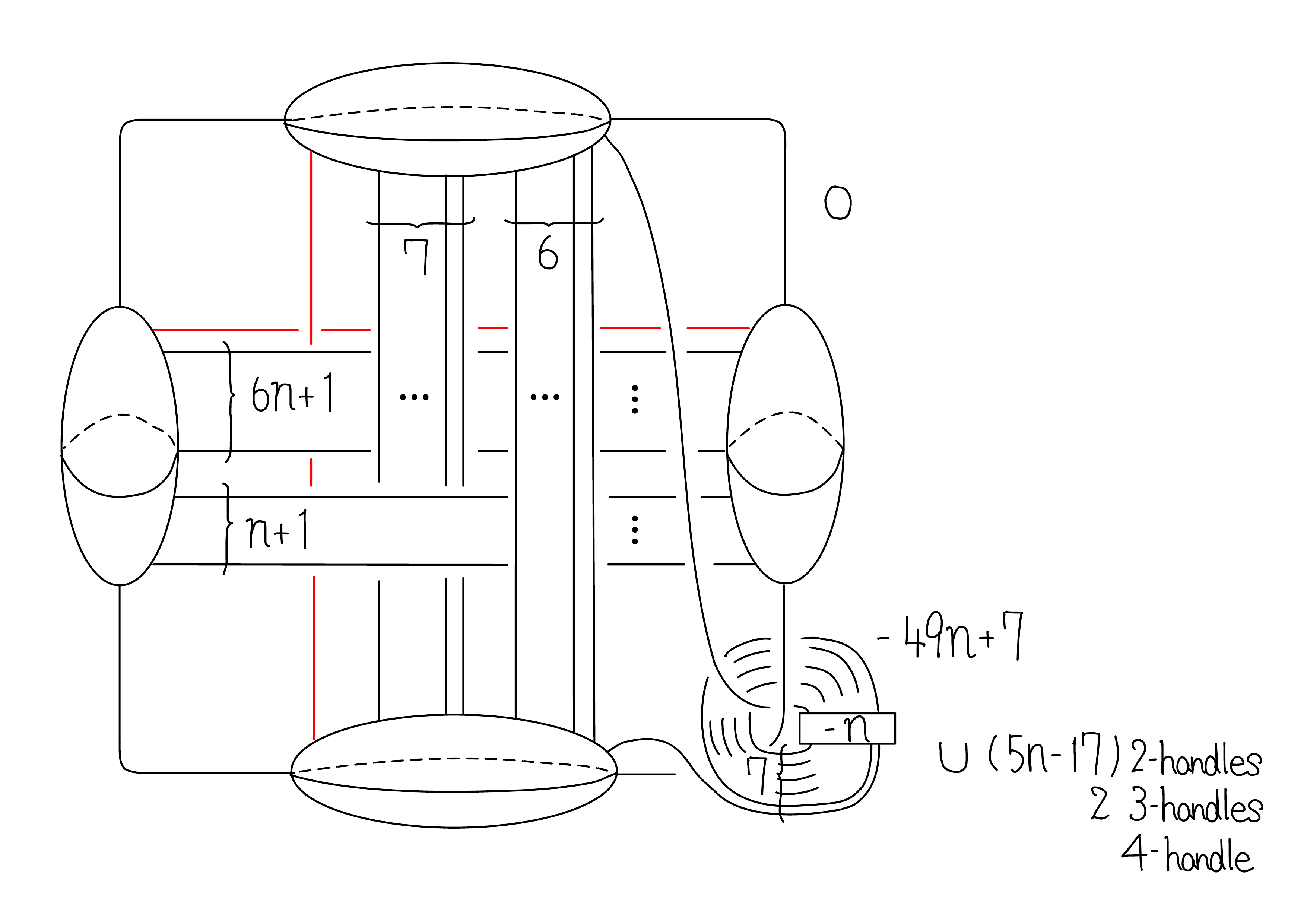} 
 \caption{}
 \label{fig:prop.3}
\end{center}
\end{figure}

\begin{proof}
We begin with Figure \ref{fig:diagram-15}, that is, the $p=7$ case of Figure \ref{fig:diagram-0}.
Note that in the  proof of this proposition, we do not slide the 1-handles, the 0-framed 2-handle and the red 2-handles in Figure \ref{fig:diagram-15}.
Using the conversions $aba\rightarrow bab$, $bab\rightarrow aba$, $(ab)^5\rightarrow (a^3b)^2aa$ and $(ab)^5\rightarrow bb(ab^3)^2$, we change the word $(ab)^{6n}$ and discover the word $a^6b^{n+1}a^7b^3(ab)^{5n-9}$ for the following conversions:
\begin{align*}
&
abababababab\underbrace{ab\cdots ab}_{12n-12}
\rightarrow
aaabaaabaaab\underbrace{ab\cdots ab}_{2n-18}\underbrace{abababababab}_{12}\underbrace{ab\cdots ab}_{10n-6}\\
&\rightarrow
aaabaaabaaab\underbrace{ab\cdots ab}_{2n-18}\underbrace{bbabbbabbbab}_{12}\underbrace{ab\cdots ab}_{10n-6}\\
&\rightarrow
aaabaaabaaab\underbrace{ab\cdots ab}_{2n-18}\underbrace{bbabbbabbbab}_{12}\underbrace{abababababab}_{12}\underbrace{ab\cdots ab}_{10n-18}\\
&\rightarrow
aaa\underline{b}aaab\underline{aaa}b\underbrace{\underline{a}b\cdots \underline{a}b}_{2n-18}\underbrace{bb\underline{a}bbb\underline{a}bbba\underline{b}}_{12}\underbrace{aa\underline{b}aaa\underline{b}
ab\underline{a}bb}_{12}\underbrace{ab\cdots ab}_{10n-18}\\
&\sim
\underbrace{a\cdots a}_{6}\underbrace{b\cdots b}_{n+1}\underbrace{aaaaaaa}_{7}bbb\underbrace{ab\cdots ab}_{10n-18}\\
\end{align*}
Here, we omit underlined letters in the last conversion.
Therefore, in Figure \ref{fig:diagram-15}, we perform handle slides and isotopies indicated by these conversions and obtain Figure \ref{fig:diagram-20}.
In Figure \ref{fig:diagram-20}, we repeat the deformation in Figure \ref{fig:diagram-7} (b) $(n-2)$ times, so that we can change $(5n-10)$ vertical 2-handles and $(5n-10)$ horizontal ones into $(2n-4)$ vertical ones and $(8n-16)$ horizontal ones.  
Then, we omit $(2n-4)$ vertical and $(2n-14)$ horizontal knots with $(-1)$ framings and get Figure \ref{fig:prop.3}. 
\end{proof}

\begin{prop}
\label{prop:prop-3.4}
  For \(n \geq 24\), the elliptic surface $E(n)_8$ admits a handle decomposition in Figure \ref{fig:prop.4}. In this figure, the cusp neighborhood that contains the red 2-handles is isotopic to a tubular neighborhood of a cusp fiber in $E(n)_8$.
\end{prop}

\begin{figure}[htbp]
\begin{center}
  \includegraphics[width=11cm]{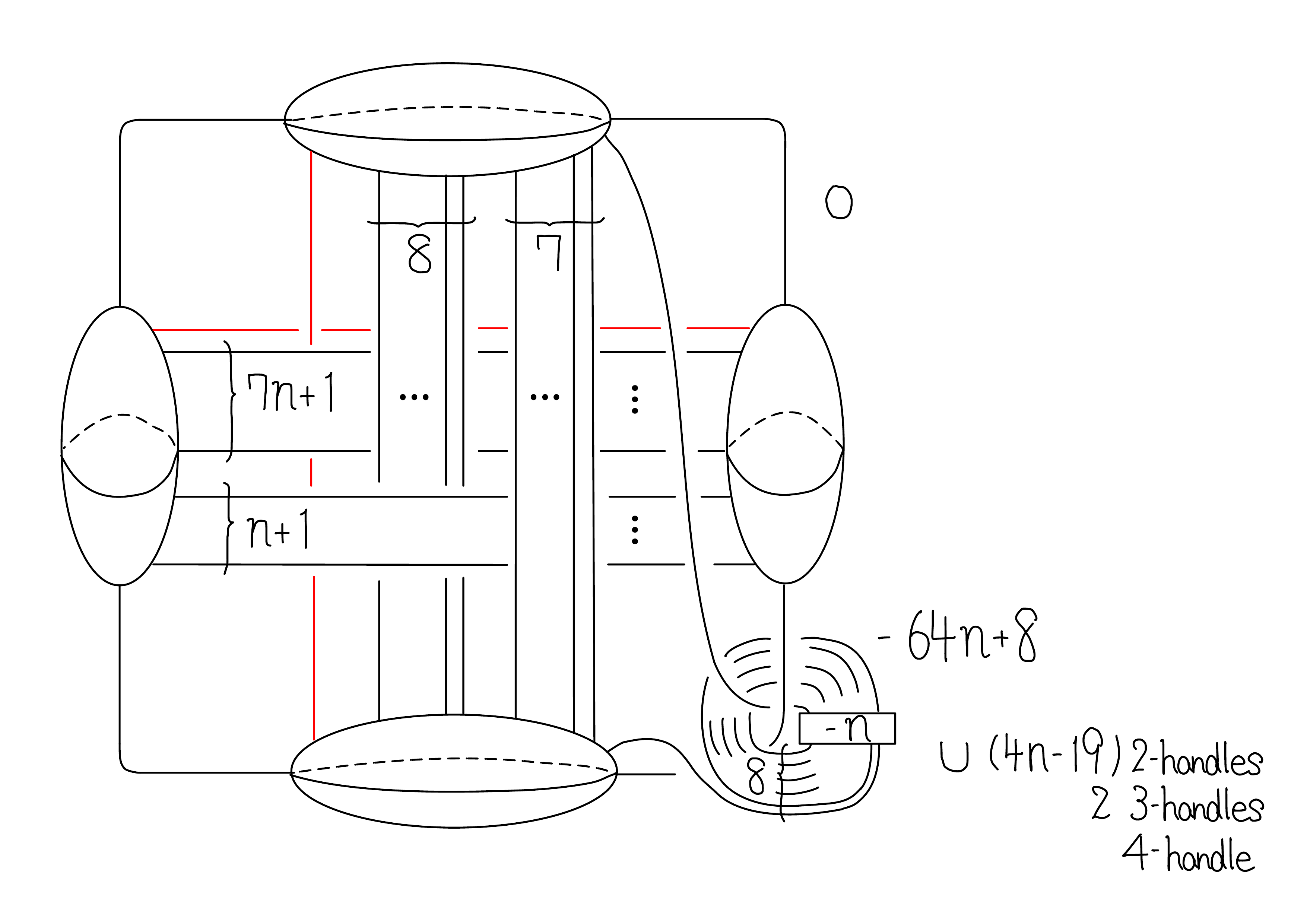} 
 \caption{}
 \label{fig:prop.4}
\end{center}
\end{figure}

\begin{proof}
We begin with Figure \ref{fig:diagram-21}, that is, the $p=8$ case of Figure \ref{fig:diagram-0}.
Note that in the proof of this proposition, we do not slide the 1-handles, the 0-framed 2-handle and the red 2-handles in Figure \ref{fig:diagram-21}.
Using the conversions $aba\rightarrow bab$, $bab\rightarrow aba$ and $(ab)^5\rightarrow bb(ab^3)^2$, we change the word $(ab)^{6n}$ and discover the word $a^7b^{n+1}a^8b^4(ab)^{5n-6}$ for the following conversions:
\begin{align*}
&abababababab\underbrace{ab\cdots ab}_{12n-12}
\rightarrow
aabaaabaabab\underbrace{abababababab}_{12}\underbrace{abababababab}_{12}\underbrace{ab\cdots ab}_{12n-36}\\
&\rightarrow
aabaaabaabab\underbrace{bbabbbabbbab}_{12}\underbrace{bbabbbabbbab}_{12}\underbrace{ab\cdots ab}_{2n-48}\underbrace{ab\cdots ab}_{10n+12}\\
&=
aabaaabaabab\underbrace{bbabbbabbbab}_{12}\underbrace{bbabbbabbbab}_{12}\underbrace{ab\cdots ab}_{2n-48}\underbrace{abababababab}_{12}\underbrace{abababababab}_{12}\underbrace{ab\cdots ab}_{10n-12}\\
&\rightarrow
aa\underline{b}aaa\underline{b}aab\underline{a}b\underbrace{bb\underline{a}bbb\underline{a}bbb\underline{a}b}_{12}\underbrace{bb\underline{a}bbb\underline{a}bbb\underline{a}b}_{12}\underbrace{\underline{a}b\cdots \underline{a}b}_{2n-48}\underbrace{b\underline{a}bbb\underline{a}ba\underline{b}aa\underline{b}}_{12}\underbrace{aa\underline{b}aa\underline{b}abb\underline{a}bb}_{12}\underbrace{ab\cdots ab}_{10n-12}\\
&\sim
\underbrace{a\cdots a}_{7}\underbrace{b\cdots b}_{n+1}\underbrace{a\cdots a}_{8}bbbb\underbrace{ab\cdots ab}_{10n-12}
\end{align*}
Here, we omit underlined letters in the last conversion.
Therefore, in Figure \ref{fig:diagram-21}, we perform handle slides and isotopies described by these conversions and obtain Figure \ref{fig:diagram-27}.
In Figure \ref{fig:diagram-27}, we repeat the deformation in Figure \ref{fig:diagram-7} (b) $(n-2)$ times, so that we can change $(5n-10)$ vertical 2-handles and $(5n-10)$ horizontal ones into $(2n-4)$ vertical ones and $(8n-16)$ horizontal ones.  
Then, we omit $(2n-1)$ vertical and $(n-10)$ horizontal knots with $(-1)$ framings and get Figure \ref{fig:prop.4}. 
\end{proof}

\begin{proof}[Proof of Theorem \ref{thm:theorem1.2}]
The elliptic surface $E(n)_p$ considered in the above propositions has a handle decomposition in Figure \ref{fig:diagram-28}.
First, by performing the handle slide indicated in Figure \ref{fig:diagram-28}, we obtain Figure \ref{fig:diagram-29}.
In this figure, we repeat a handle slide similar to the above slide and get Figure \ref{fig:diagram-30}.
Now, the top-left picture in Figure \ref{fig:diagram-31} is included in Figure \ref{fig:diagram-30}. 
Doing the handle slide indicated by the arrow in the top-left diagram, we get the bottom-left. 
In this diagram, we repeat a similar handle slide and obtain the bottom-right and the top-right diagrams.
Furthermore, in Figure \ref{fig:diagram-30}, we repeat a process similar to that of changing the top-left diagram to the top-right in Figure \ref{fig:diagram-31} and obtain Figure \ref{fig:diagram-32}.
Next, we perform a handle slide indicated in Figure \ref{fig:diagram-32} and repeat a similar handle slide $(p-1)$ times.
We get Figure \ref{fig:diagram-33}. 
Then, we focus on Figure \ref{fig:diagram-34} (a) which is a part of Figure \ref{fig:diagram-33}.
We perform the handle slide shown in Figure \ref{fig:diagram-34} (a) and obtain Figure \ref{fig:diagram-34} (b).
By repeating such a handle slide $(p-1)$ times, we get the picture of Figure \ref{fig:diagram-34} (c).
Here, we omit the red $(-2)$-framed knot and obtain Figure \ref{fig:diagram-34} (d).
We repeat a process similar to that of changing Figure \ref{fig:diagram-34} (a) into Figure \ref{fig:diagram-34} (d) $(p-3)$ times, and get Figure \ref{fig:diagram-34} (e).
Then, doing the handle slide shown in Figure \ref{fig:diagram-34} (e), we obtain \ref{fig:diagram-34} (f).
Finally, in Figure \ref{fig:diagram-33}, we slide and isotope 2-handles as in Figure \ref{fig:diagram-34} repeatedly, and get Figure \ref{fig:diagram-35}.
Performing the handle slide indicated in Figure \ref{fig:diagram-35} and isotoping the diagram, we obtain Figure \ref{fig:diagram-36}.
Therefore, the elliptic surface $E(n)_p$ satisfies the assumption of Lemma \ref{thm:lemma2.1}, and then the elliptic surface $E(n)_{p,p+1}$ has a handle decomposition without 1-handles.
\end{proof}

\begin{figure}[htbp]
\begin{center}
  \includegraphics[width=13cm]{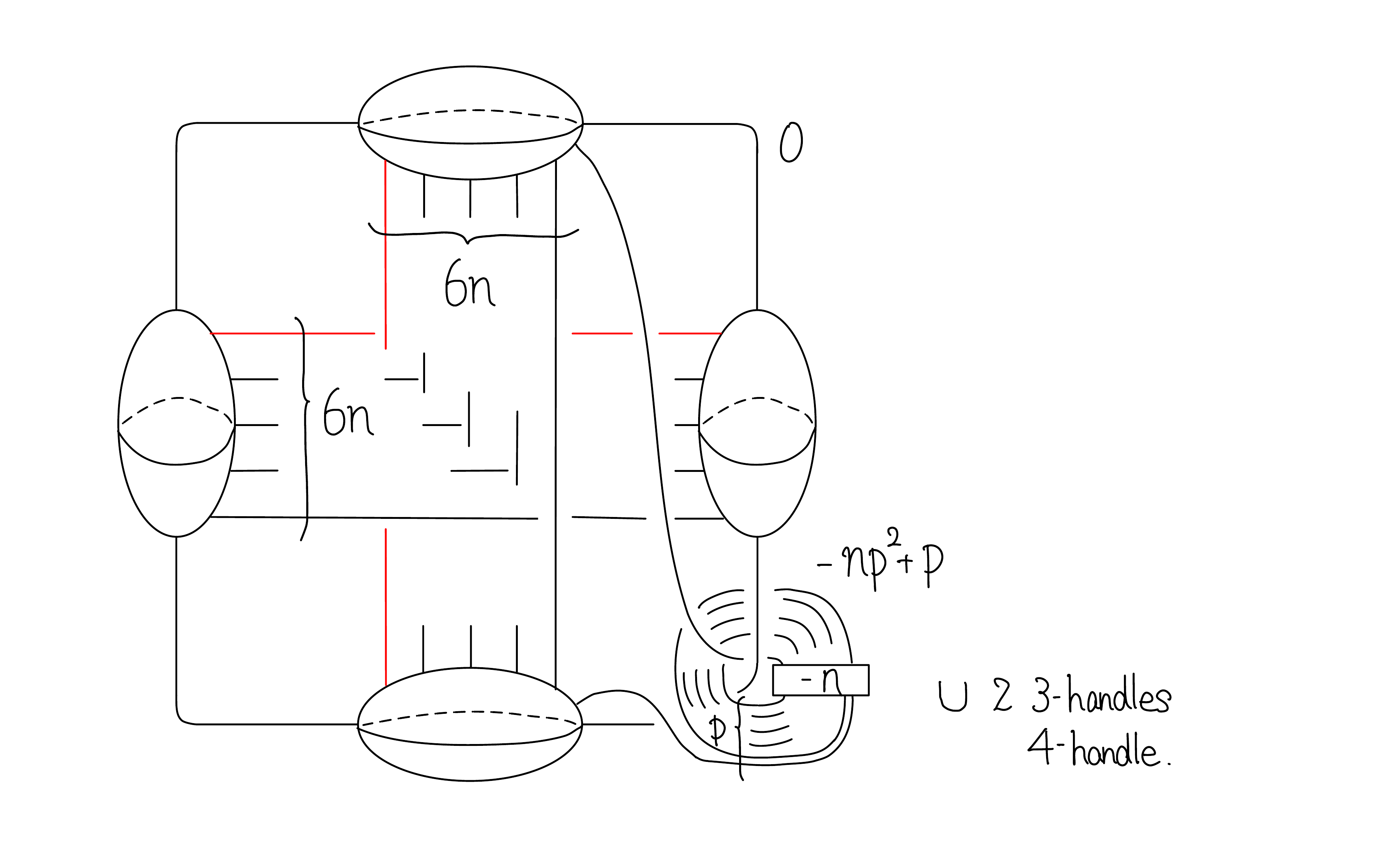} 
 \caption{}
 \label{fig:diagram-0} 
\end{center}
\end{figure}
\

\begin{figure}[htbp]
\begin{center}
  \includegraphics[width=13cm]{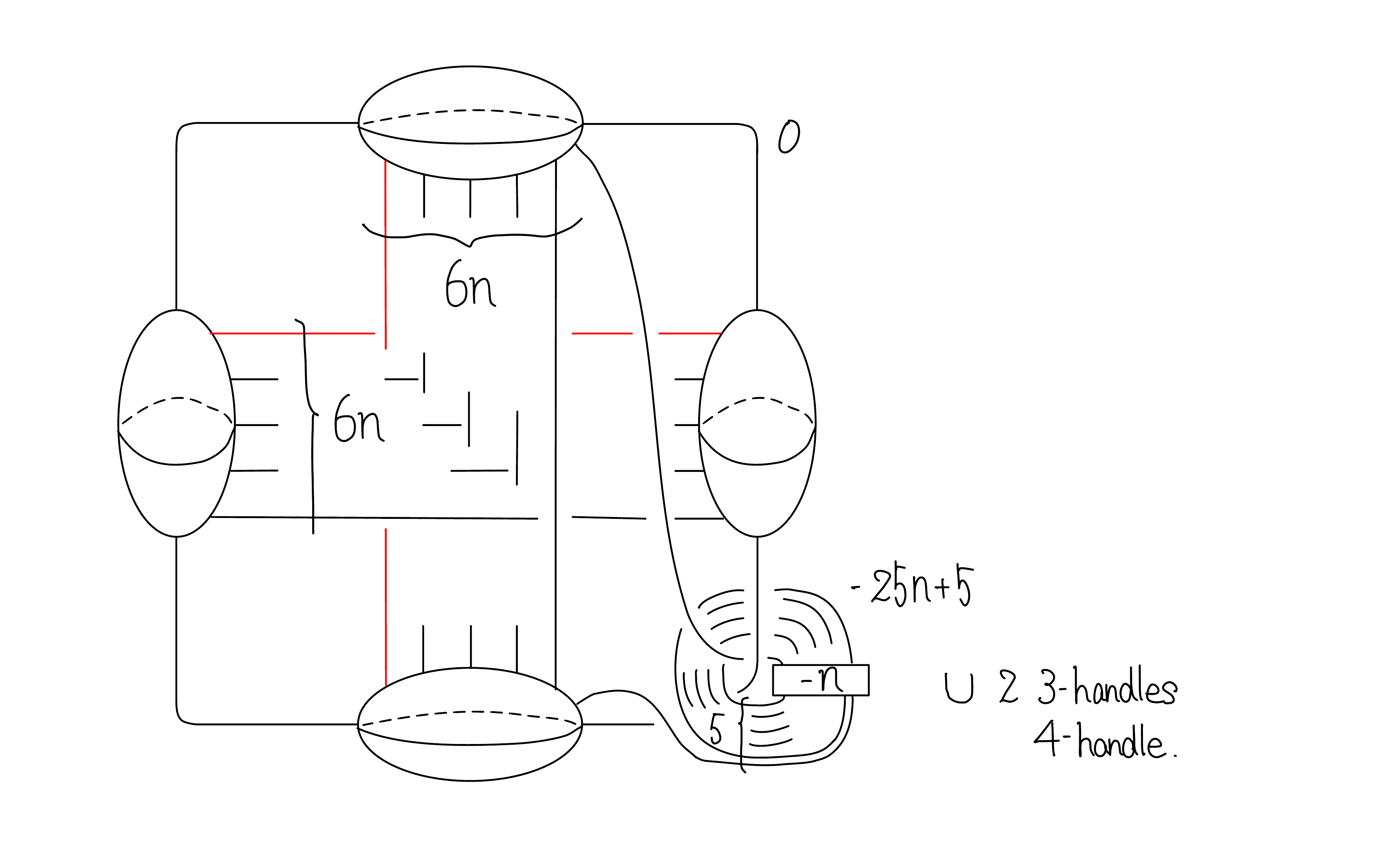} 
 \caption{}
 \label{fig:diagram-1}\end{center}
\end{figure}
\

\begin{figure}[htbp]
\begin{center}
  \includegraphics[width=12cm]{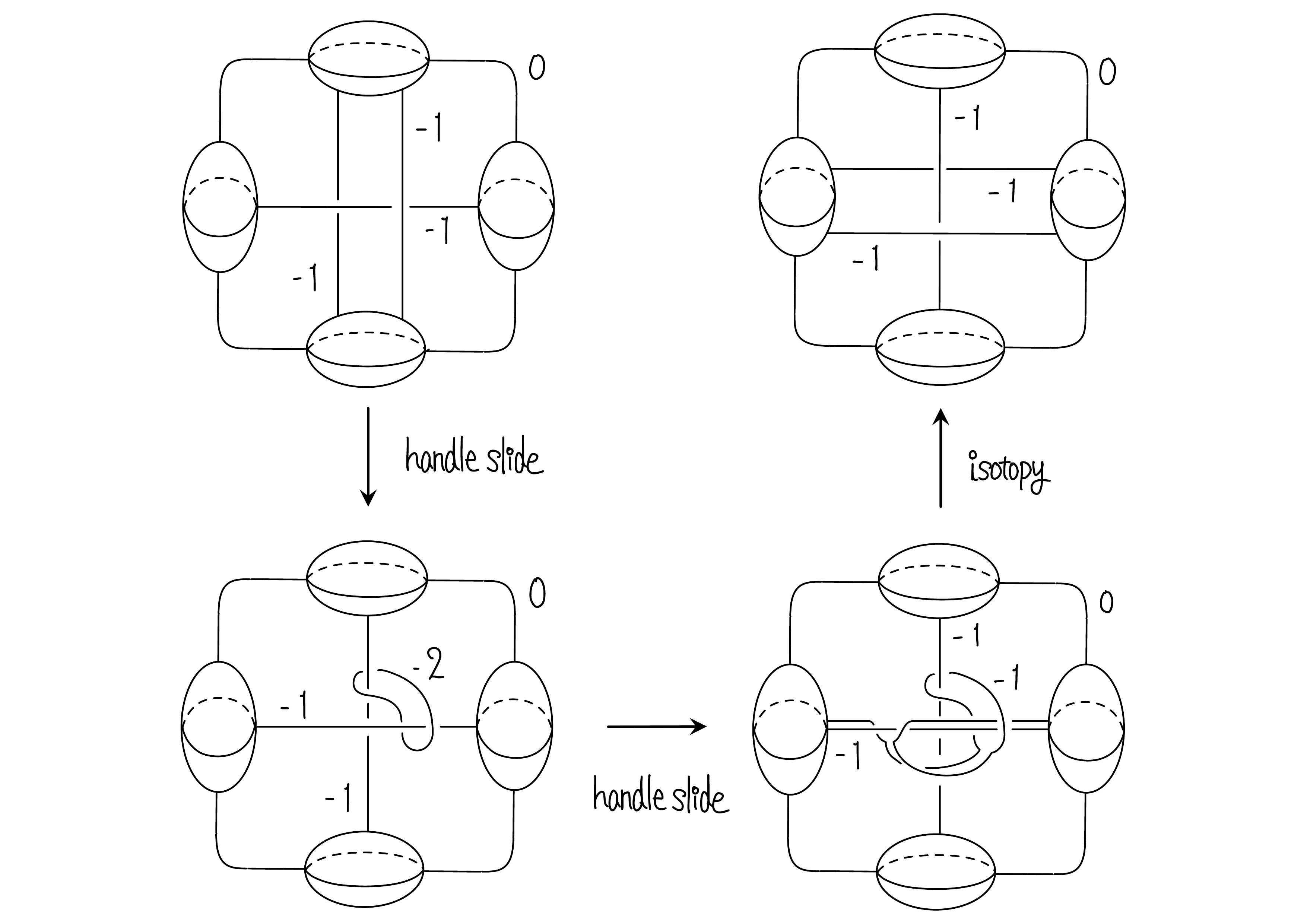} 
 \caption{}
  \label{fig:diagram-3}
\end{center}
\end{figure}

\
\begin{figure}[htbp]
\begin{center}
  \includegraphics[width=12cm]{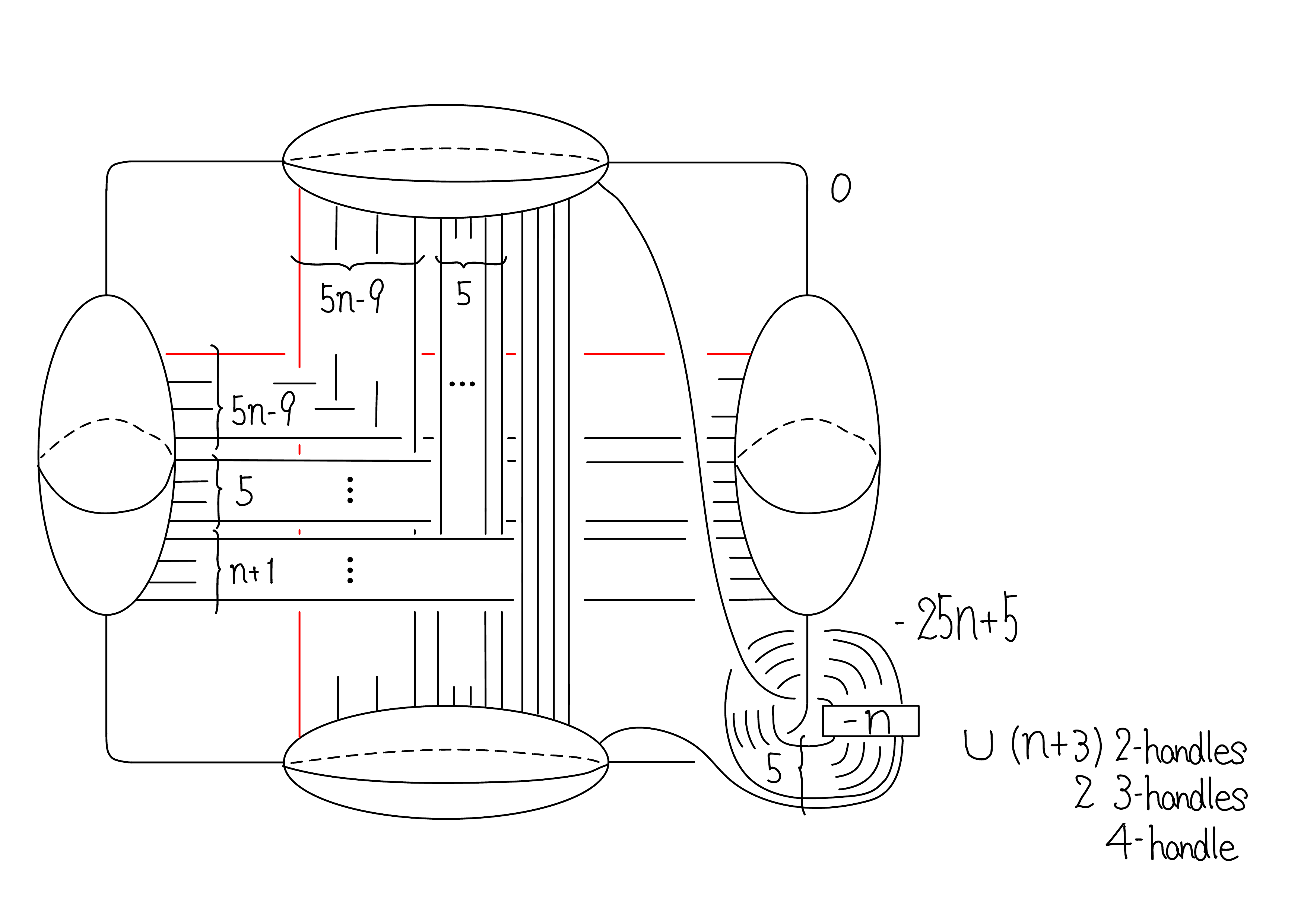} 
 \caption{}
  \label{fig:diagram-6}
\end{center}
\end{figure}
\
\begin{figure}[htbp]
\begin{center}
  \includegraphics[width=12cm]{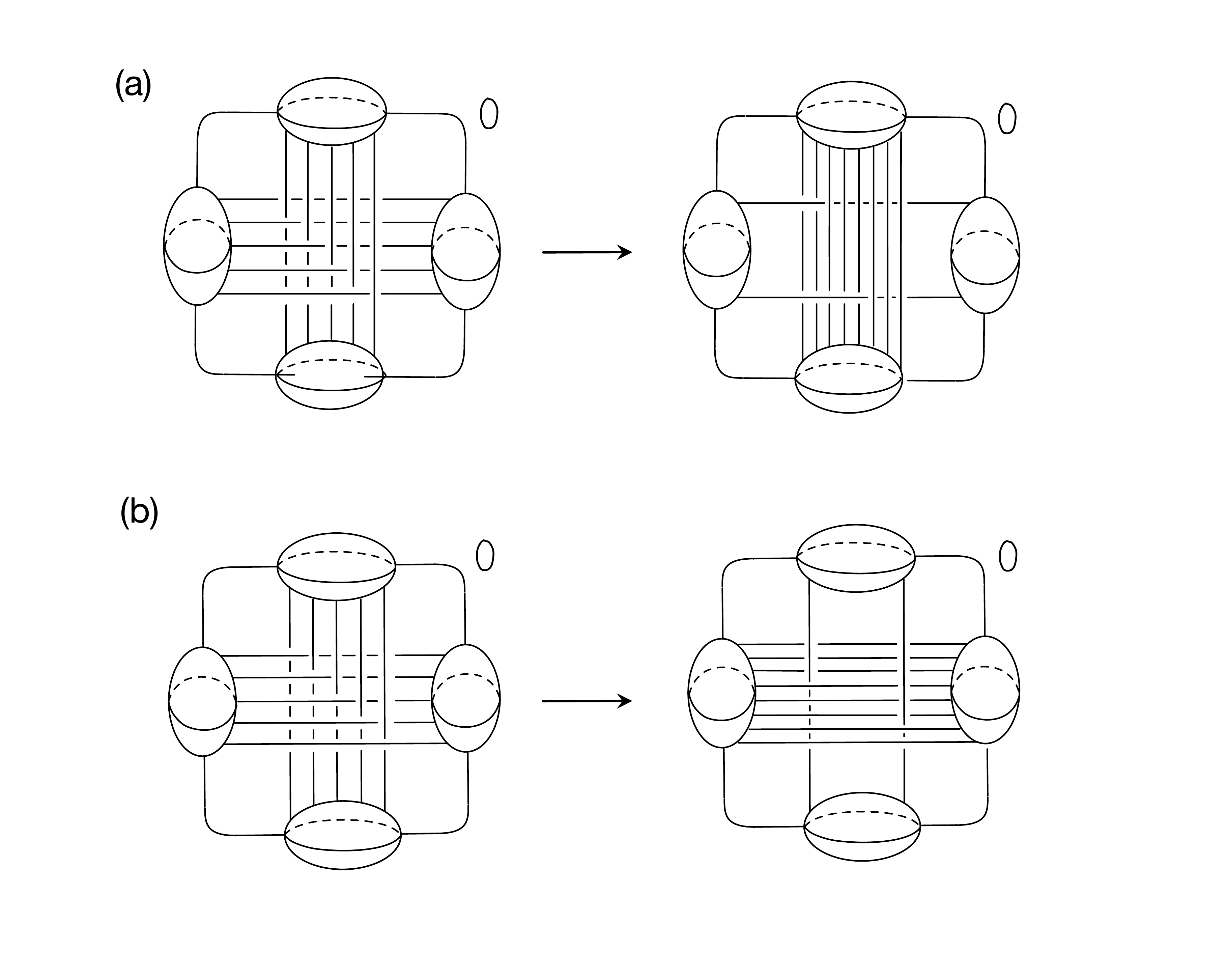} 
 \caption{}
  \label{fig:diagram-7}
\end{center}
\end{figure}
\

\begin{figure}[htbp]
\begin{center}
  \includegraphics[width=13cm]{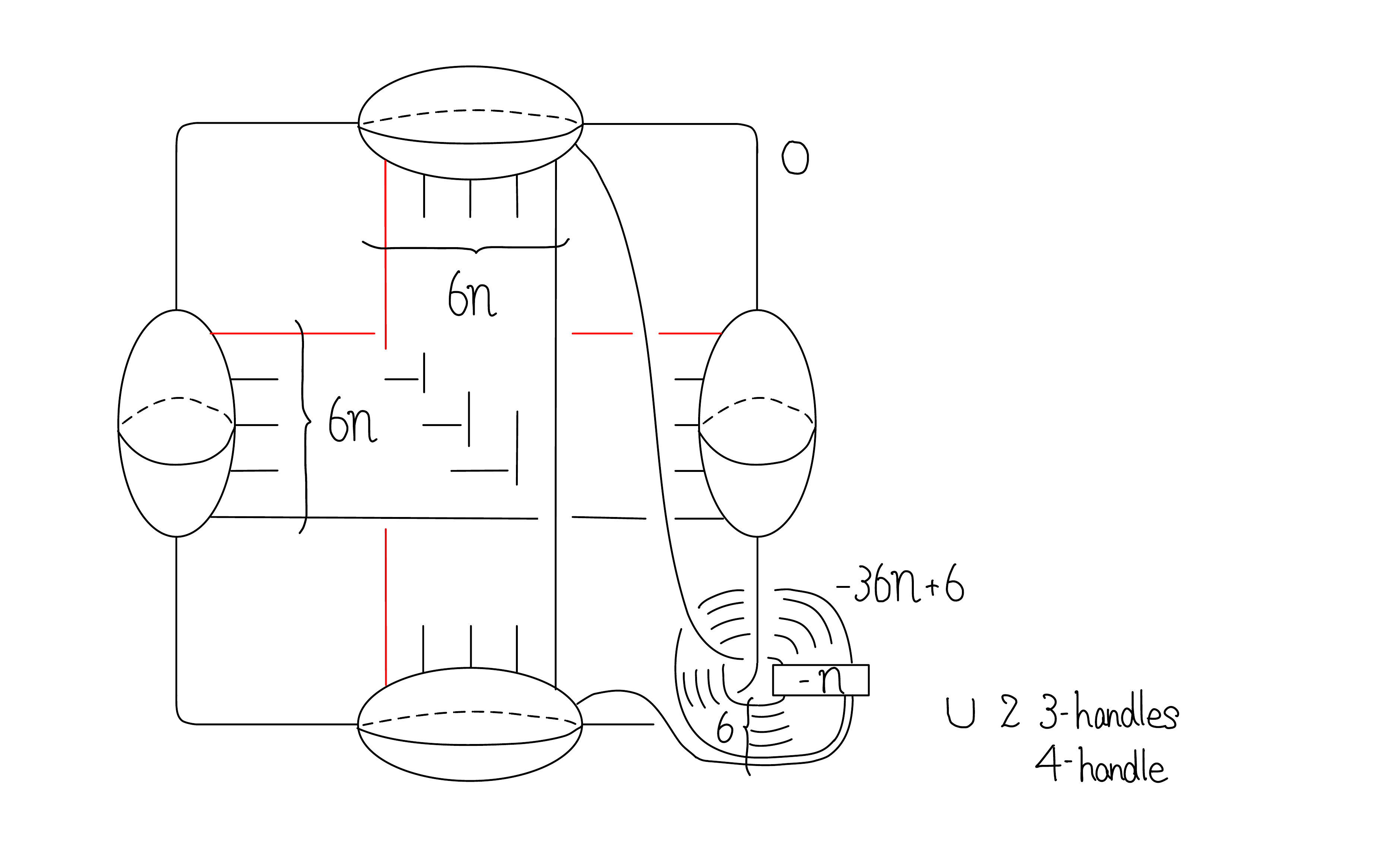} 
 \caption{} 
 \label{fig:diagram-8}
\end{center}
\end{figure}
\

\begin{figure}[htbp]
\begin{center}
  \includegraphics[width=12cm]{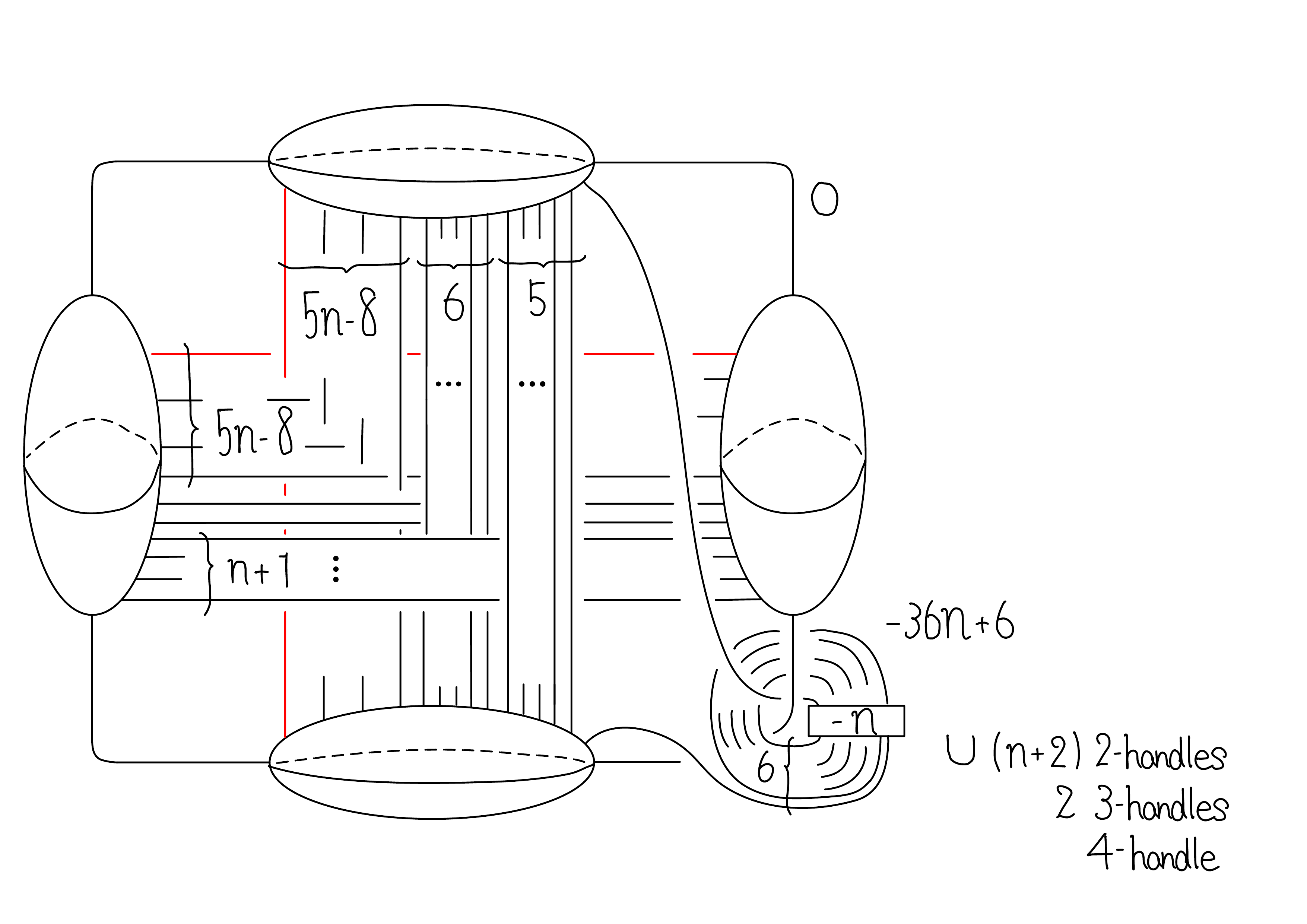} 
 \caption{}
 \label{fig:diagram-13}
\end{center}
\end{figure}
\

\begin{figure}[htbp]
\begin{center}
  \includegraphics[width=13cm]{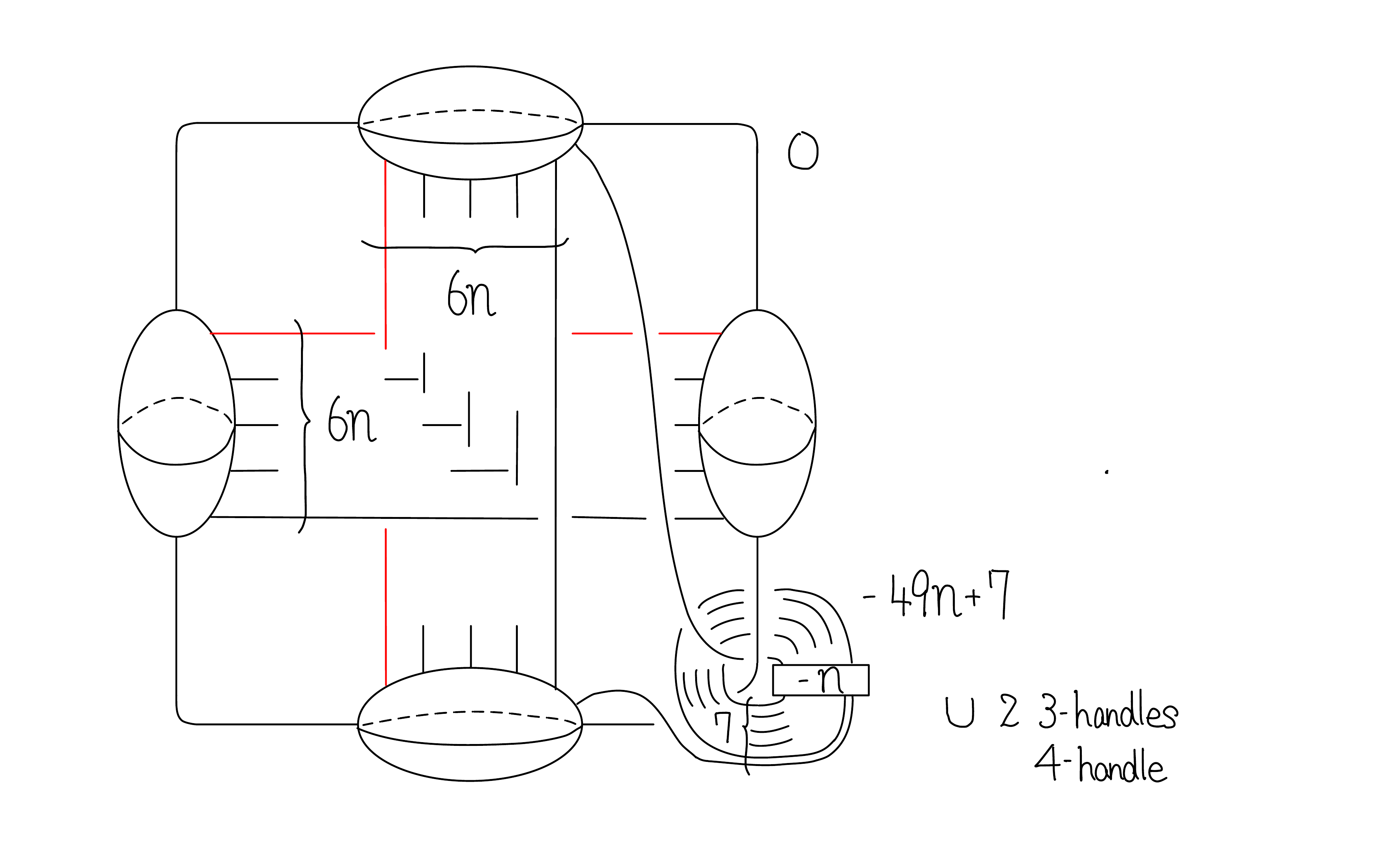} 
 \caption{}
 \label{fig:diagram-15}
\end{center}
\end{figure}
\

\begin{figure}[htbp]
\begin{center}
  \includegraphics[width=12cm]{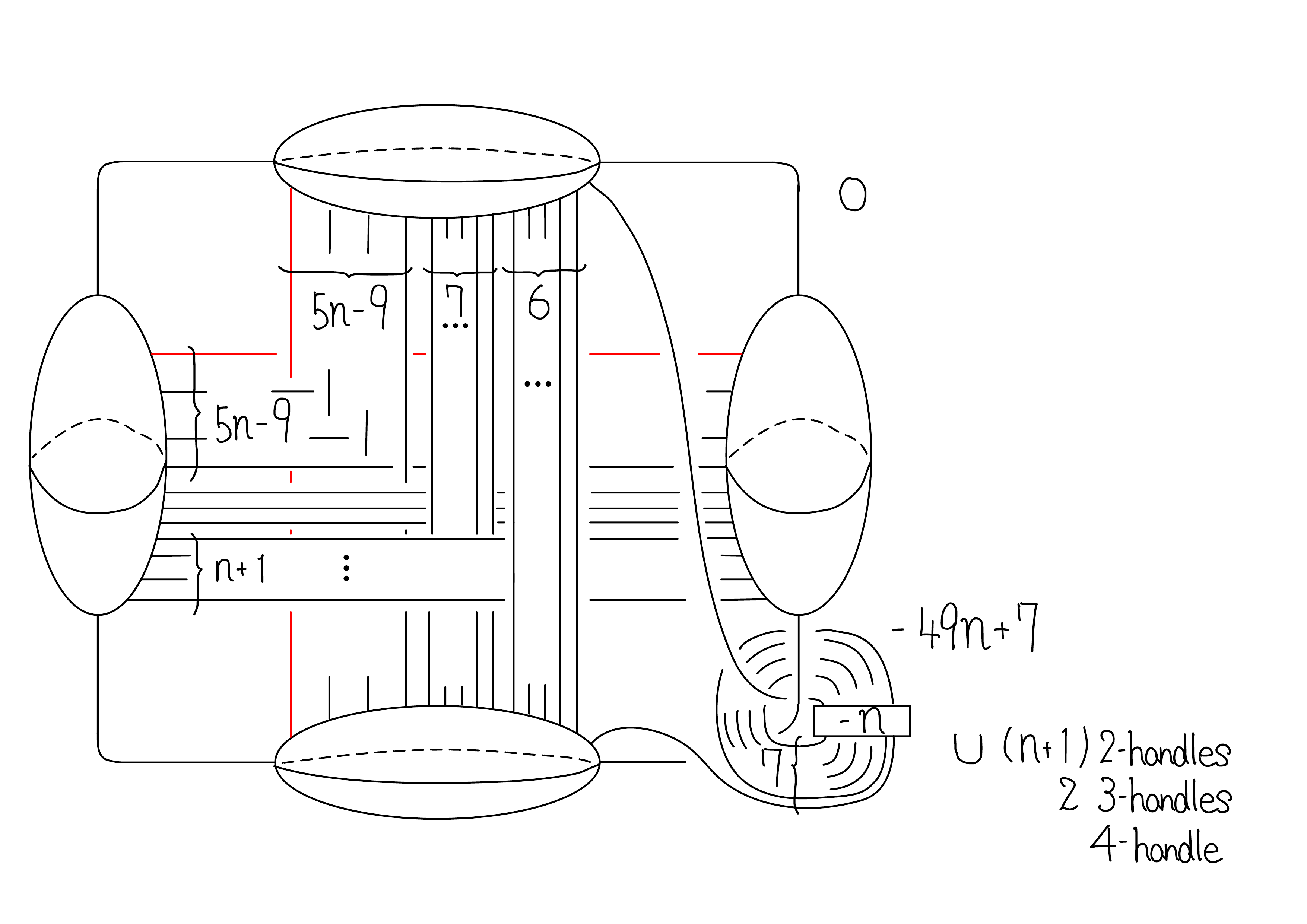} 
 \caption{}
 \label{fig:diagram-20}
\end{center}
\end{figure}
\
\begin{figure}[htbp]
\begin{center}
  \includegraphics[width=12cm]{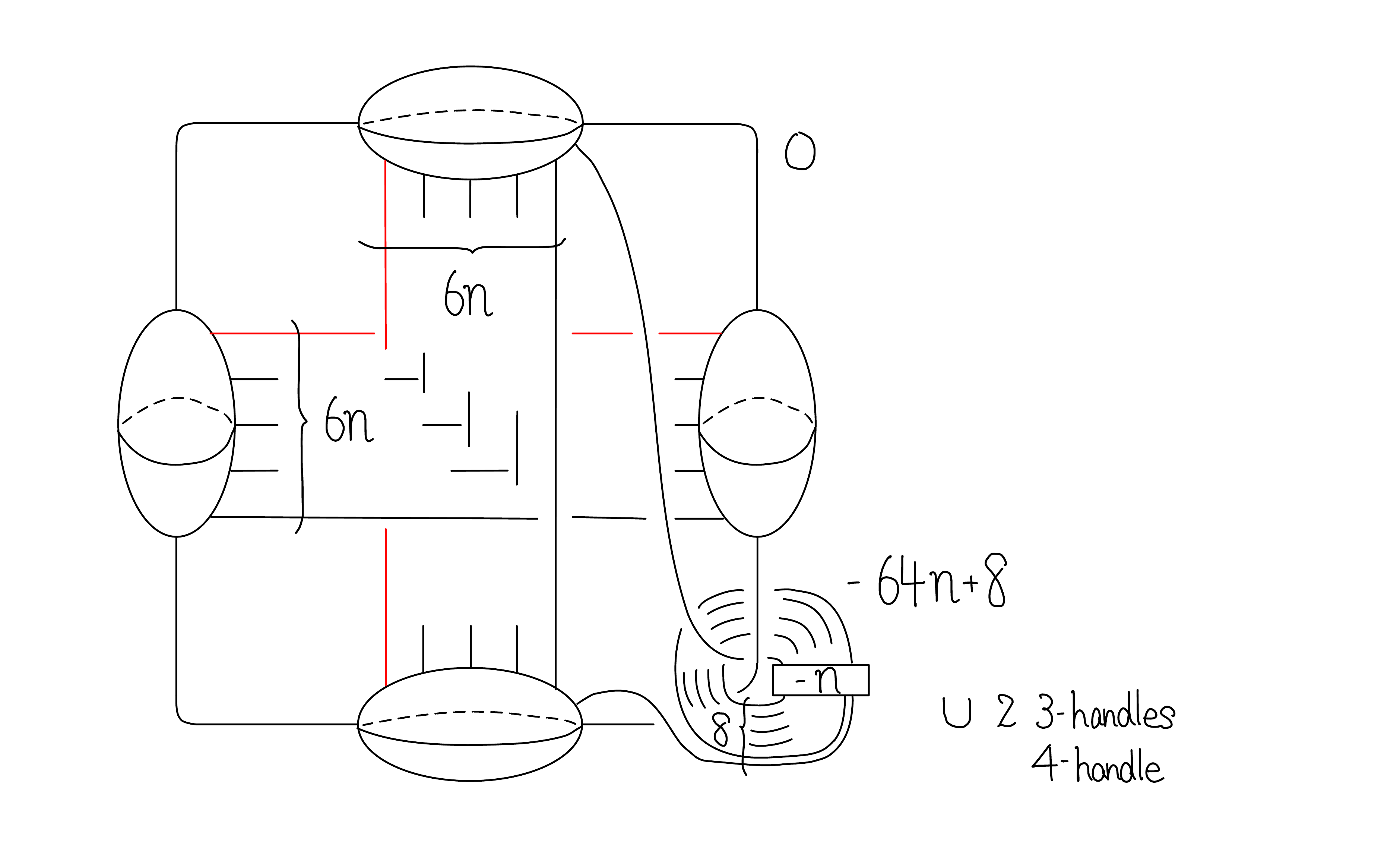} 
 \caption{}
 \label{fig:diagram-21}
\end{center}
\end{figure}
\

\begin{figure}[htbp]
\begin{center}
  \includegraphics[width=12cm]{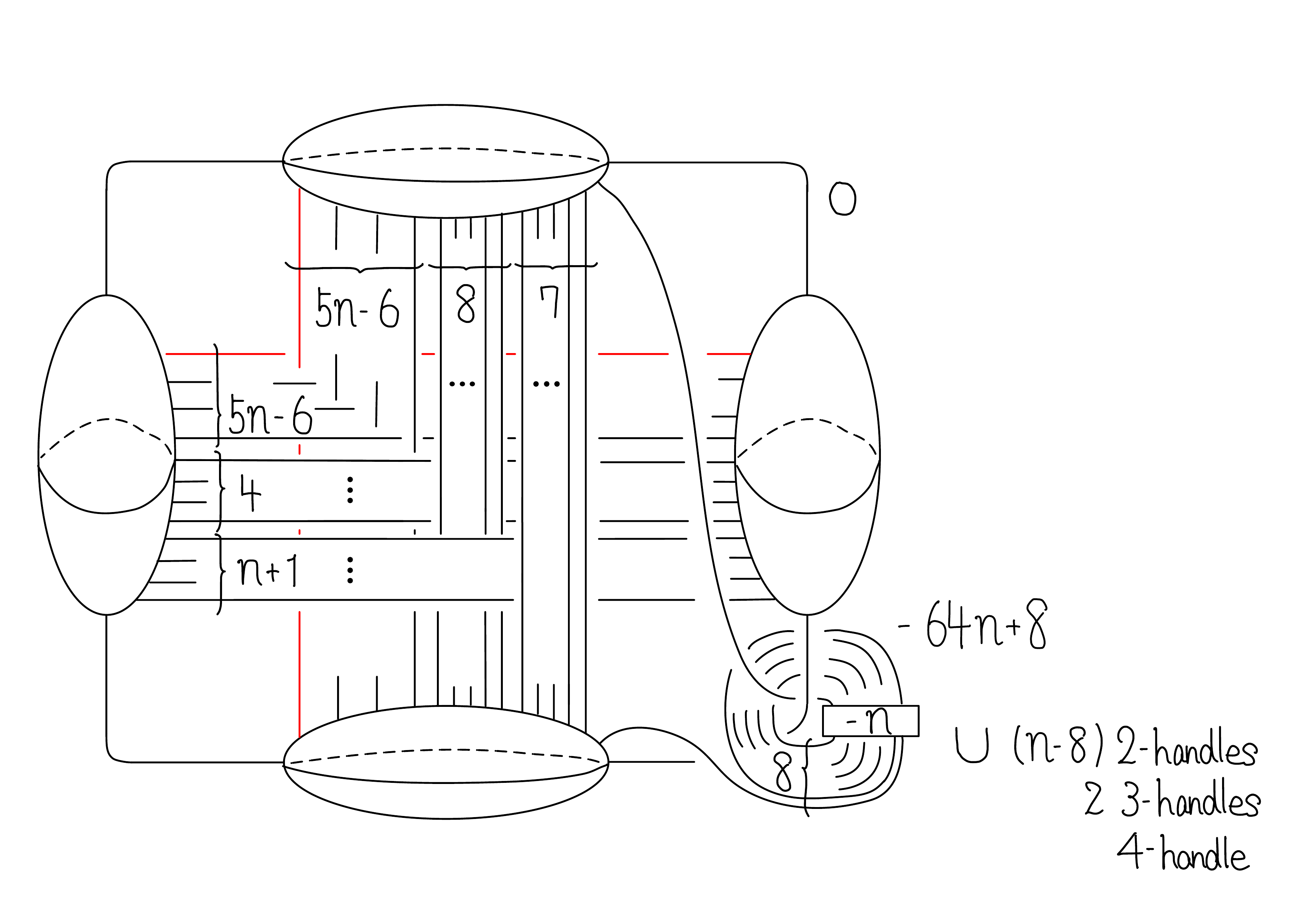} 
 \caption{}
  \label{fig:diagram-27}
\end{center}
\end{figure}
\
\
\begin{figure}[htbp]
\begin{center}
  \includegraphics[width=12cm]{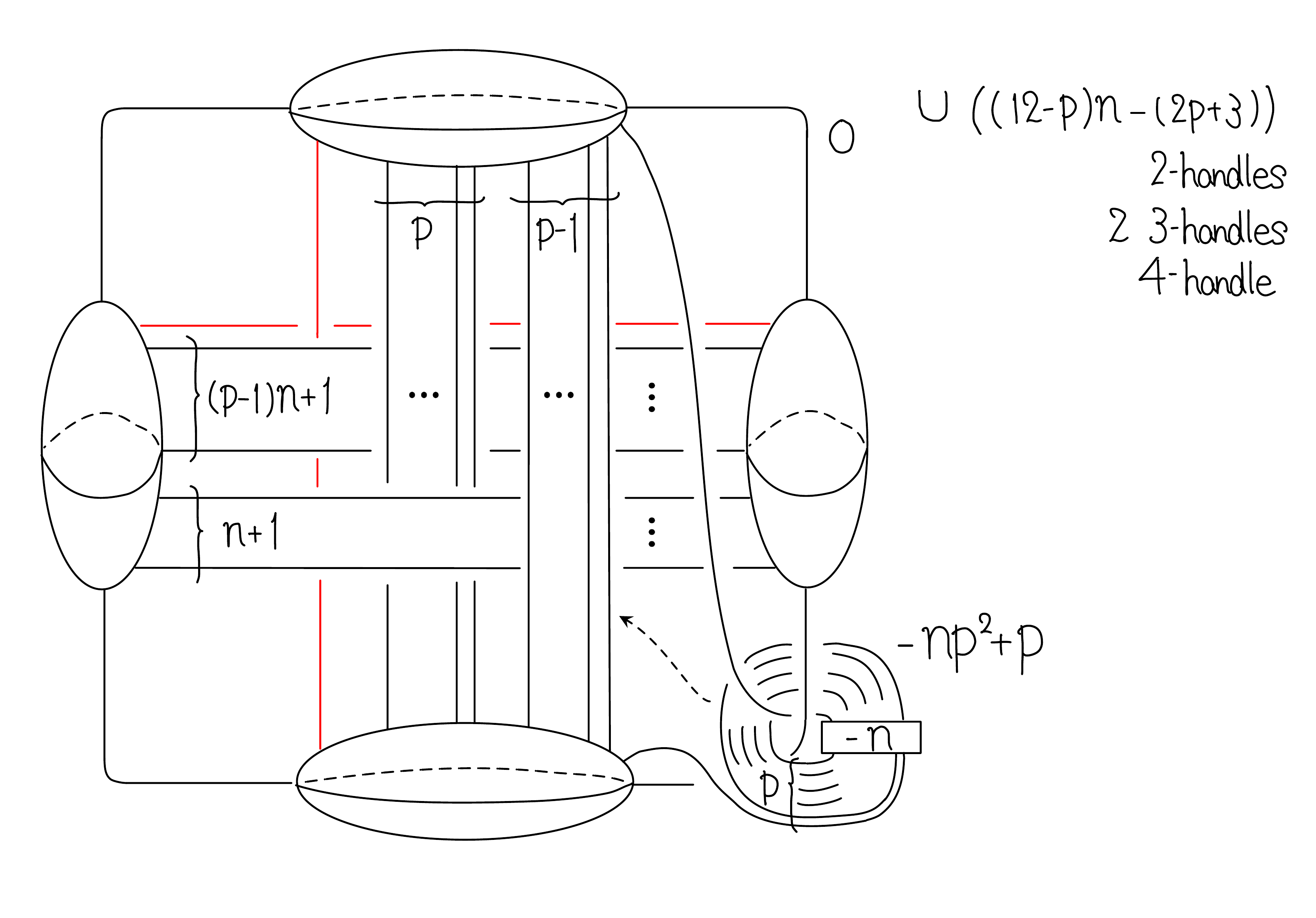} 
 \caption{}
  \label{fig:diagram-28}
\end{center}
\end{figure}
\
\begin{figure}[htbp]
\begin{center}
  \includegraphics[width=12cm]{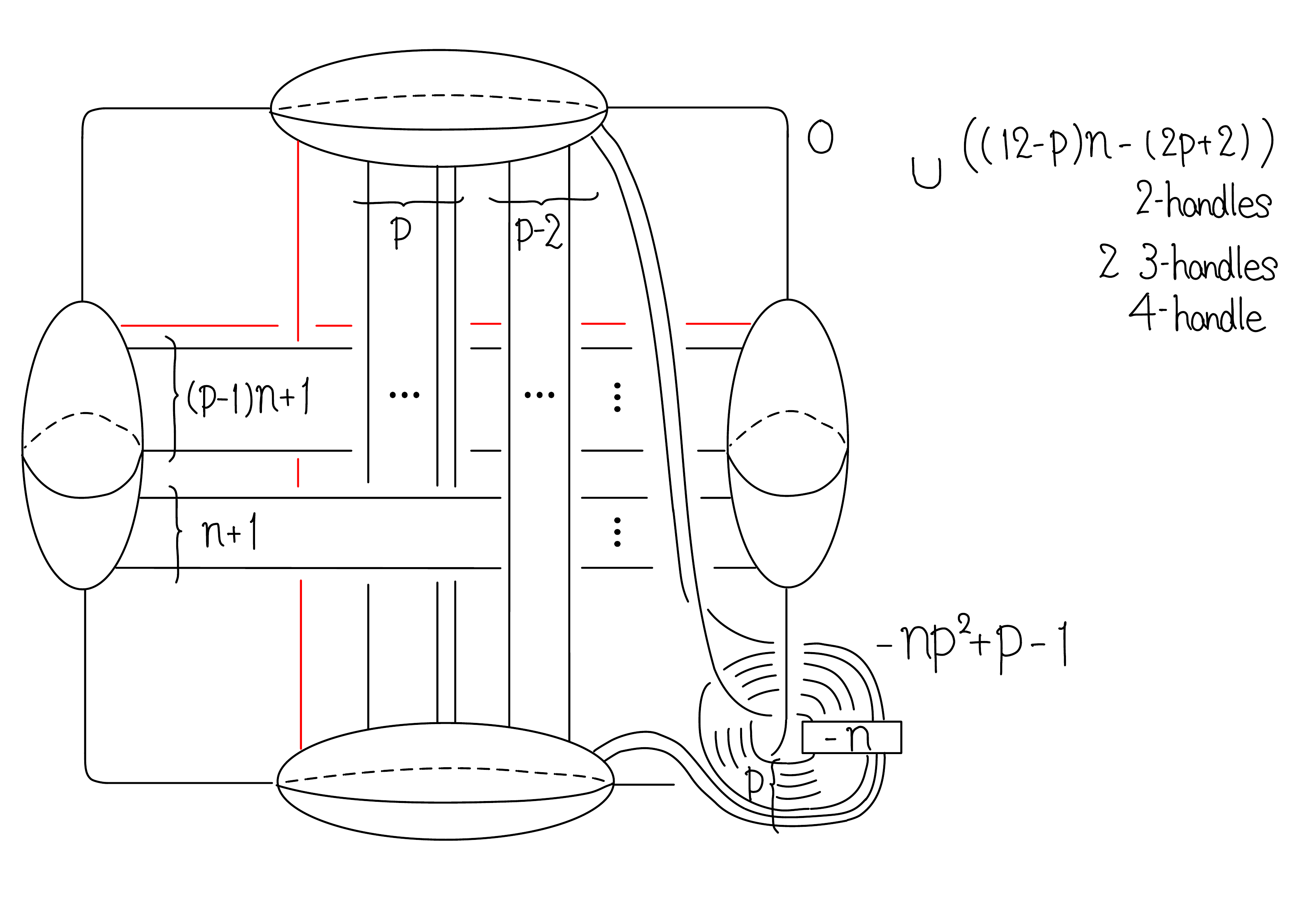} 
 \caption{}
  \label{fig:diagram-29}
\end{center}
\end{figure}
\

\begin{figure}[htbp]
\begin{center}
  \includegraphics[width=11cm]{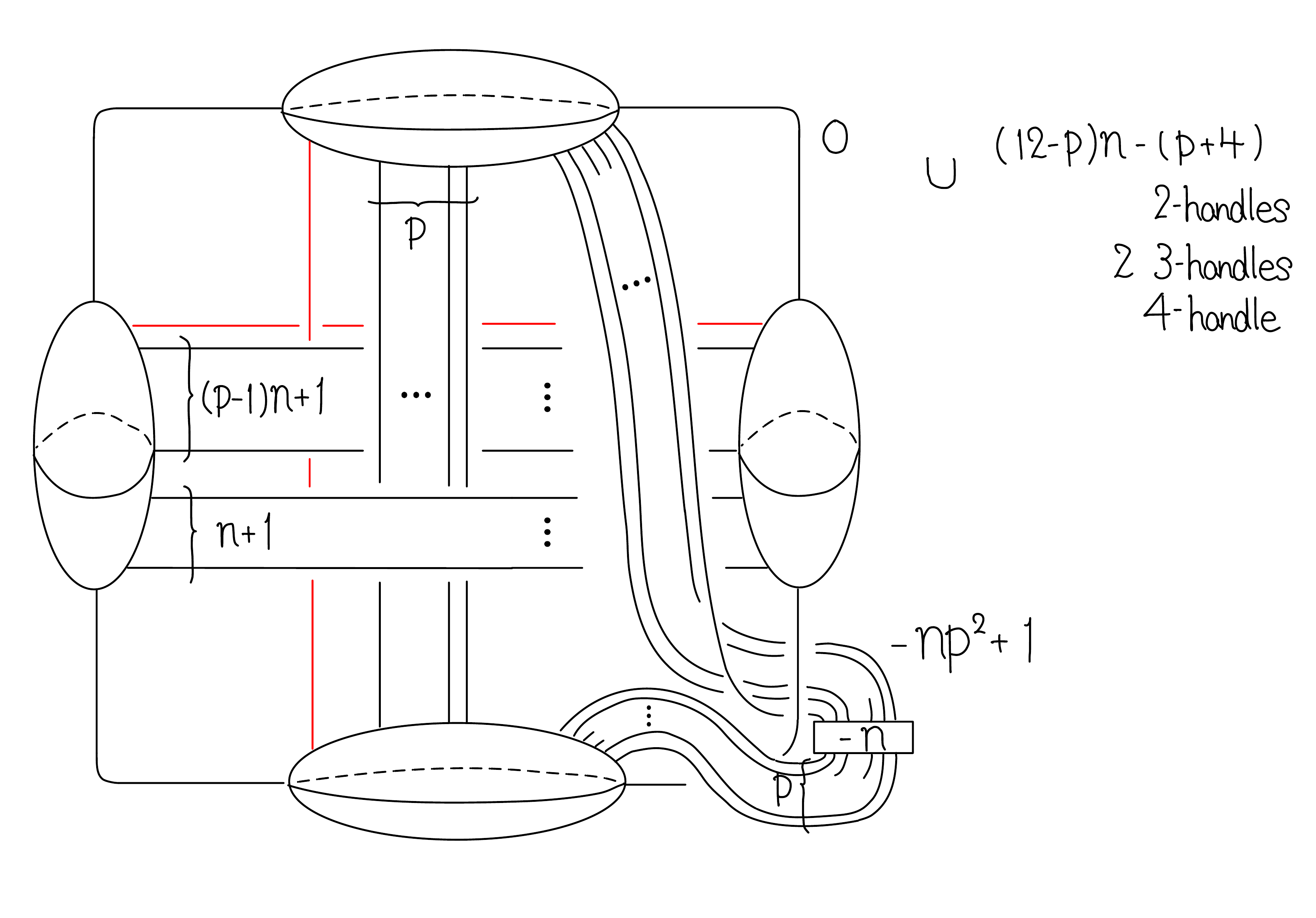} 
 \caption{}
  \label{fig:diagram-30}
\end{center}
\end{figure}
\

\begin{figure}[htbp]
\begin{center}
  \includegraphics[width=10cm]{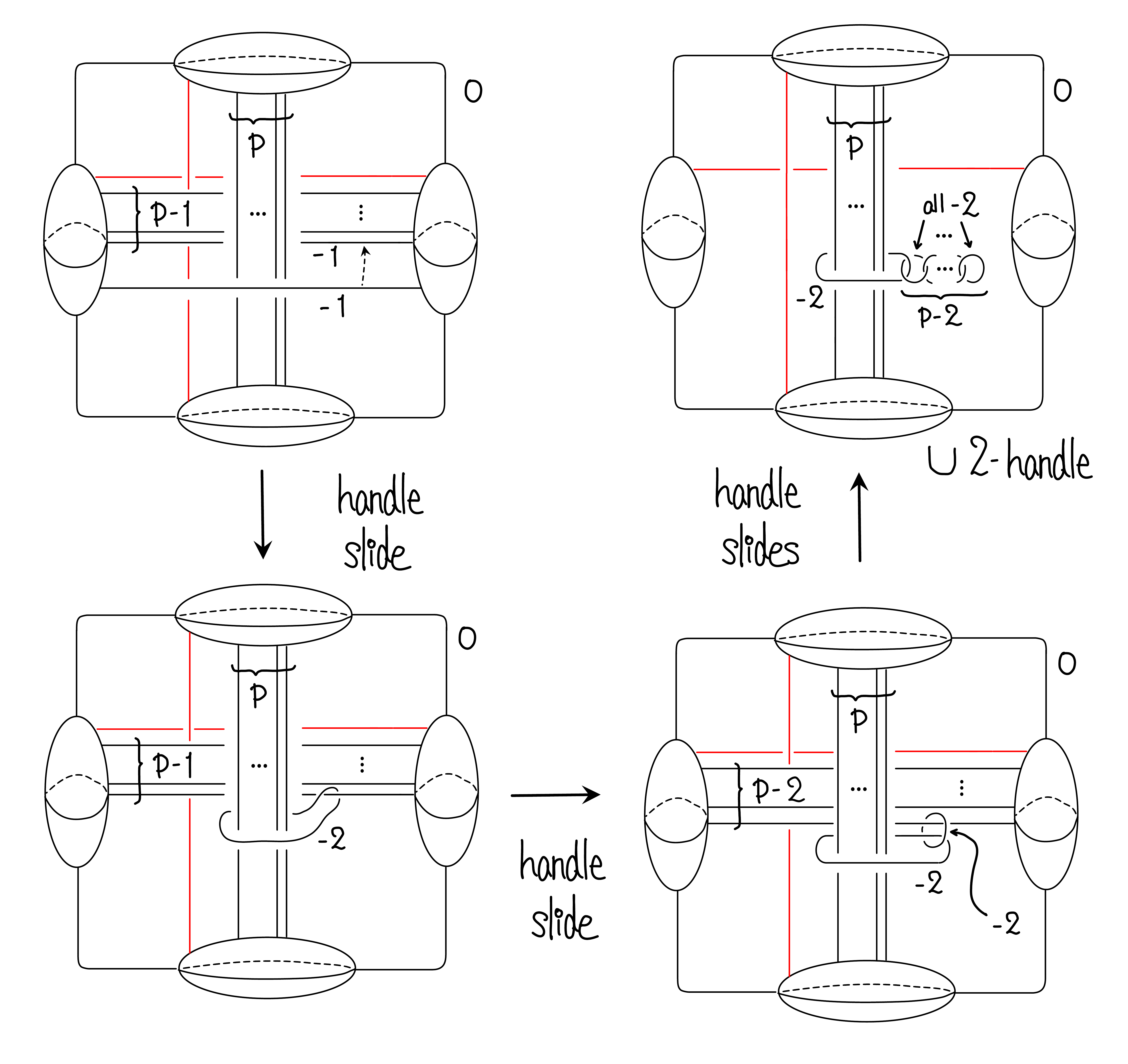} 
 \caption{}
  \label{fig:diagram-31}
\end{center}
\end{figure}
\
\begin{figure}[htbp]
\begin{center}
  \includegraphics[width=12cm]{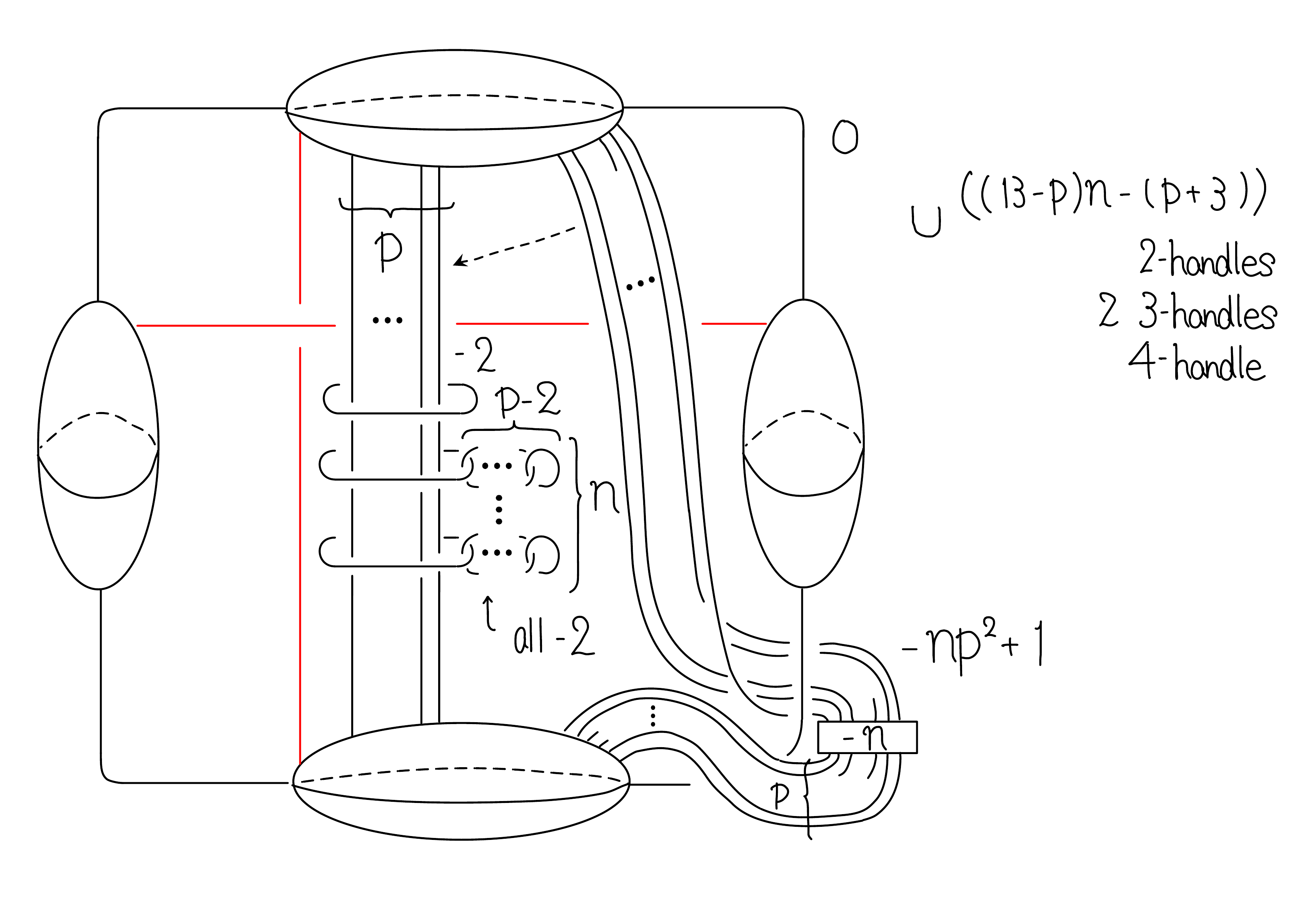} 
 \caption{}
  \label{fig:diagram-32}
\end{center}
\end{figure}
\
\begin{figure}[htbp]
\begin{center}
  \includegraphics[width=12cm]{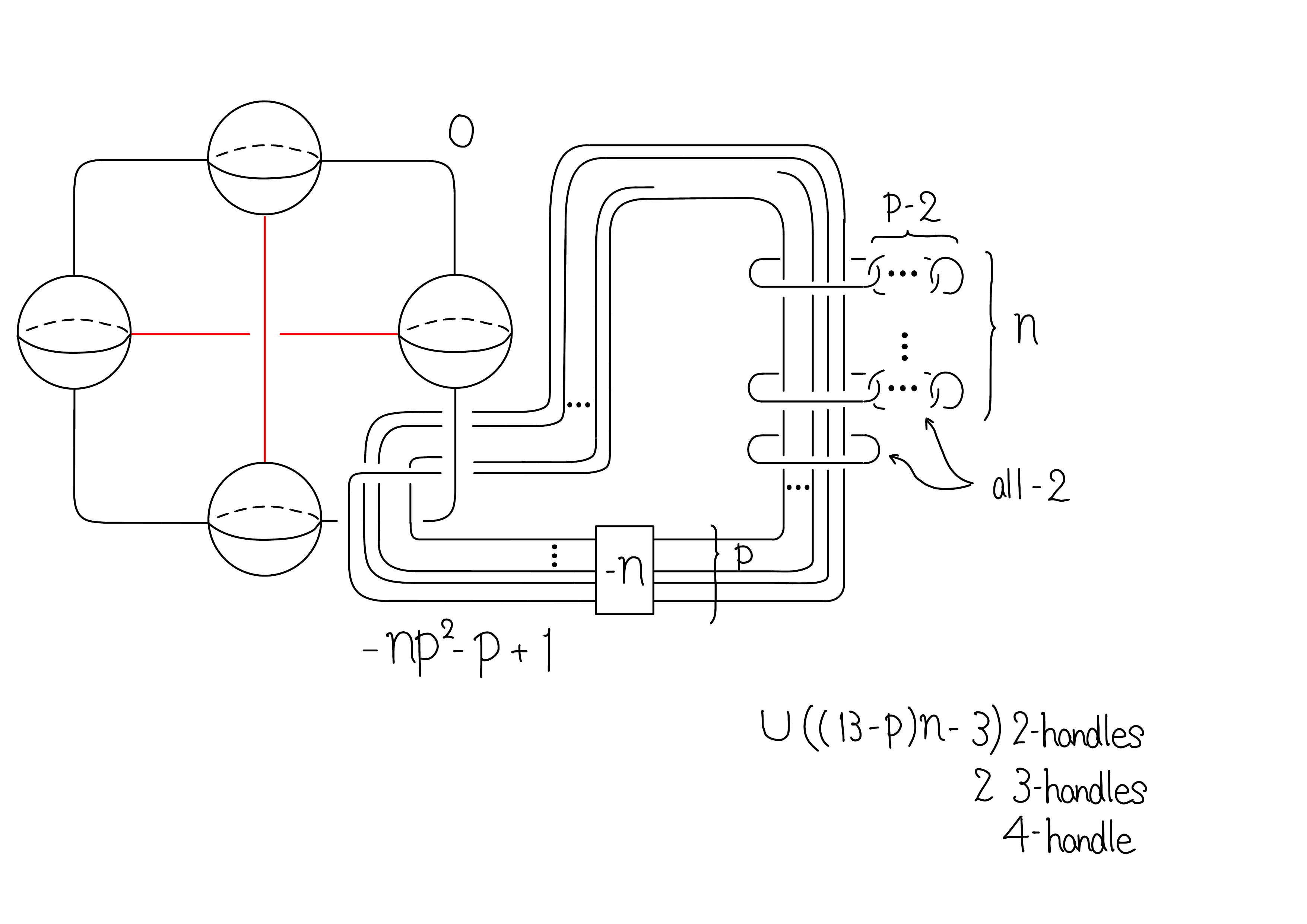} 
 \caption{}
  \label{fig:diagram-33}
\end{center}
\end{figure}
\
\begin{figure}[htbp]
\begin{center}
  \includegraphics[width=16cm]{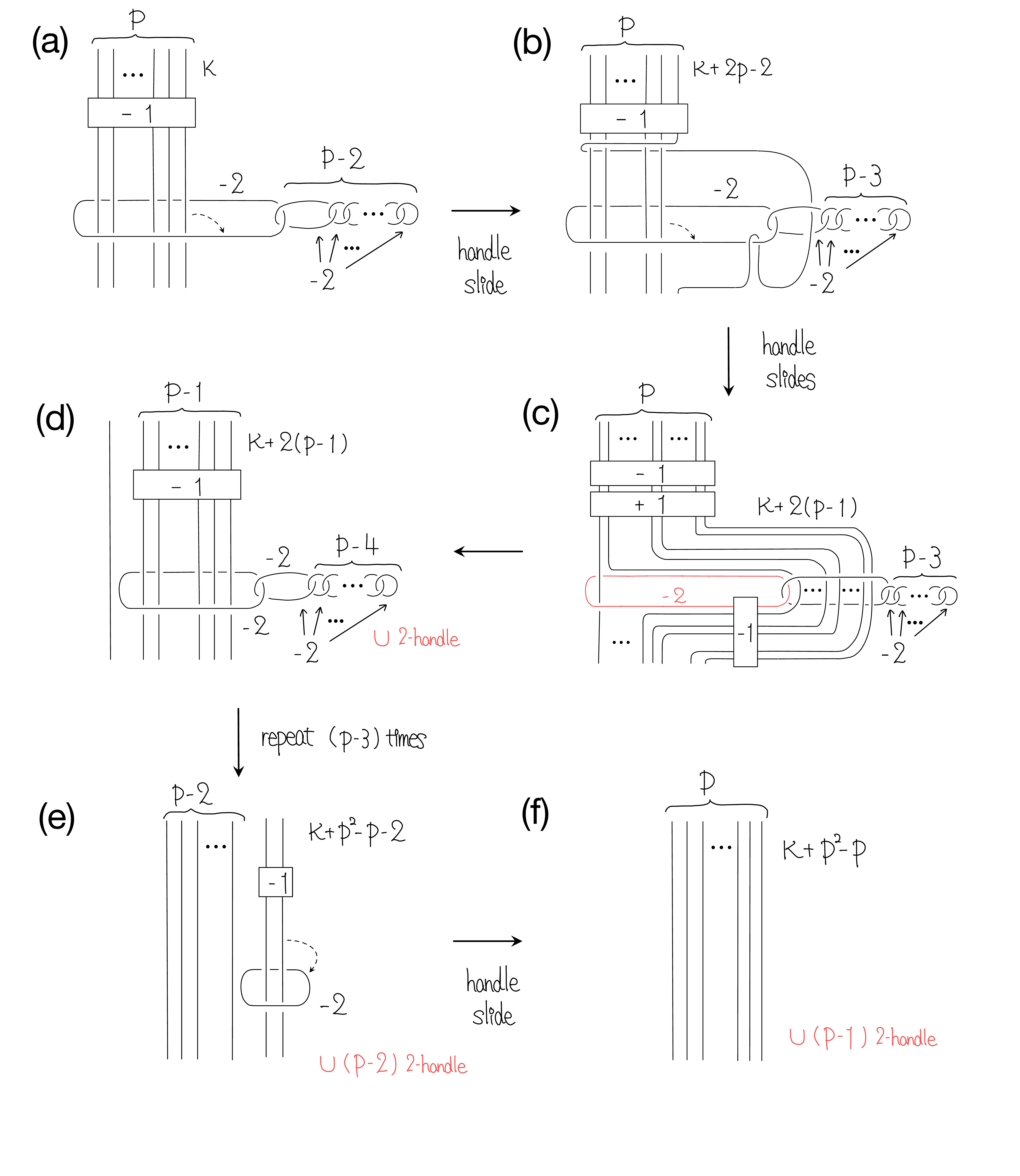} 
 \caption{}
  \label{fig:diagram-34}
\end{center}
\end{figure}
\
\begin{figure}[htbp]
\begin{center}
  \includegraphics[width=12cm]{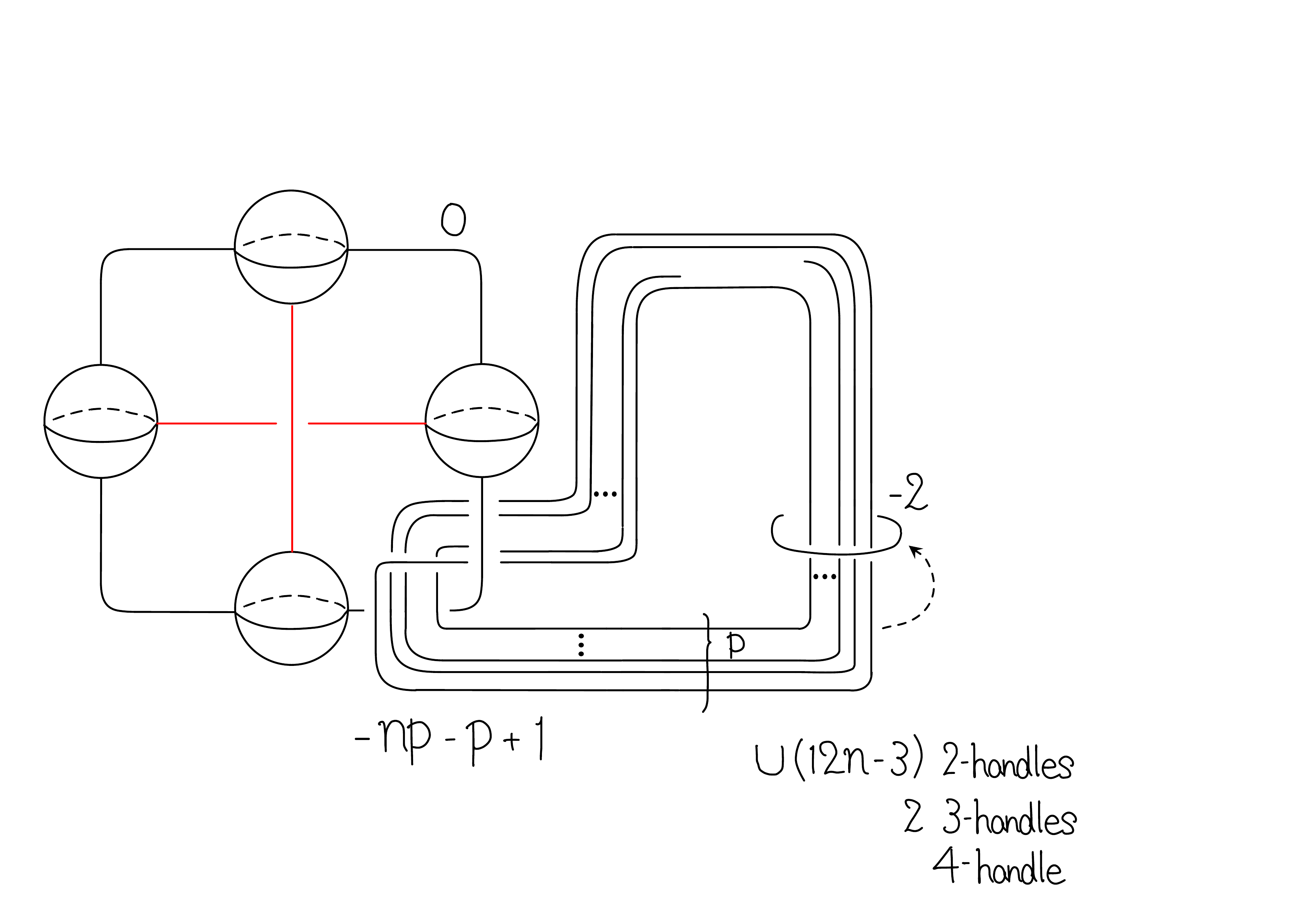} 
 \caption{}
  \label{fig:diagram-35}
\end{center}
\end{figure}
\
\begin{figure}[htbp]
\begin{center}
  \includegraphics[width=12cm]{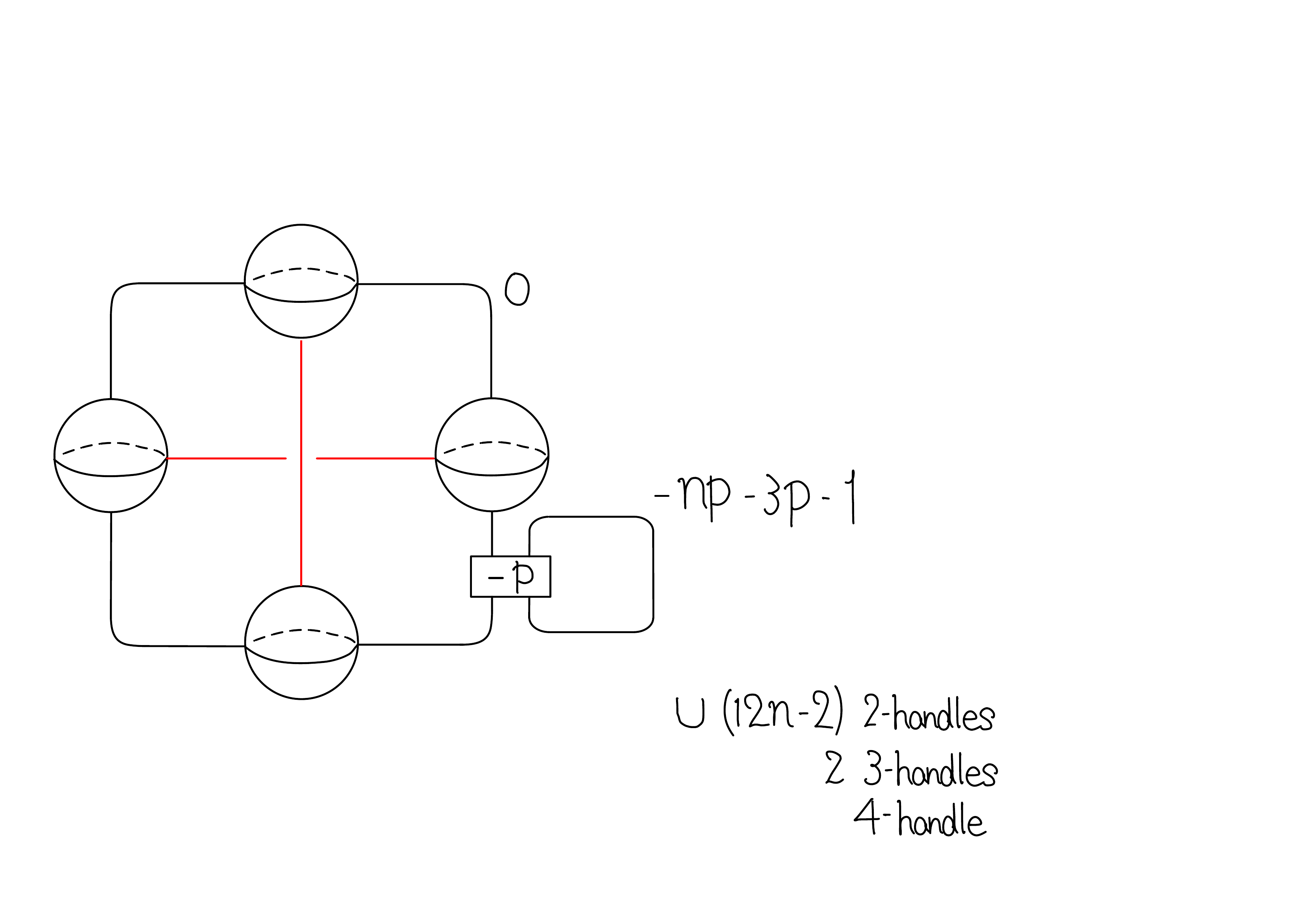} 
 \caption{}
  \label{fig:diagram-36}
\end{center}
\end{figure}
\

\section*{\footnotesize
{Department of Pure and Applied Mathematics, Graduate School of Information Science and Technology, Osaka University, 1-5 Yamadaoka, Suita, Osaka 565-0871, Japan}}

\end{document}